\documentclass[11pt]{article}
\usepackage{verbatim}
\usepackage{authblk}
\usepackage{tikz}
\usepackage{tikz-3dplot}
\usepackage{amsmath,amssymb,amsthm}
\usepackage{hyperref}
\usepackage[capitalise]{cleveref}
\usepackage[margin=1in]{geometry}
\usepackage{fancyhdr}
\usepackage{mathtools}
\usepackage{bm}
\usepackage[mathscr]{euscript}
\usepackage{xcolor}
\usepackage{enumitem}
\usepackage{lipsum}
\usepackage{lineno}
\usepackage{tikz}
\usepackage{tikz-cd}
\usepackage{multicol}
\usepackage{tikz}
\usepackage{tikz-3dplot}

\numberwithin{equation}{section}
\newtheorem{theorem}{Theorem}[section]
\newtheorem{corollary}[theorem]{Corollary}
\newtheorem{definition}[theorem]{Definition}
\newtheorem{example}[theorem]{Example}
\newtheorem{lemma}[theorem]{Lemma}

\newtheorem{proposition}[theorem]{Proposition}
\newtheorem{remark}[theorem]{Remark}

\allowdisplaybreaks
\newcommand\hide[1]{}

\newcommand{\Z}{{\mathbb Z}}

\newcommand{\R}{{\mathbb R}}

\newcommand{\mc}[1]{\mathcal{#1}}

\usepackage[all,2cell]{xy}
	\usepackage{amssymb, hyperref}
	\usepackage[all,2cell]{xy}
	\usepackage{hyperref}
	\usepackage{array}
	\usepackage[capitalise]{cleveref}

\begin{document}
		
\title{Toric Schubert Varieties in Partial Flag Varieties}

\author[1]{Mahir Bilen Can}
\author[2]{Arpita Nayek}
\author[3]{Pinakinath Saha}
\affil[1]{\small{Tulane University, New Orleans, USA, mahirbilencan@gmail.com}}
\affil[2]{\small{Jaypee Institute of Information Technology, Noida, Uttar Pradesh, nayekarpita@gmail.com}}
\affil[3]{\small{Indian Institute of Technology Delhi, New Delhi, pinakinath@iitd.ac.in}}
		
\maketitle

\begin{abstract}
In this article, we investigate the toric Schubert varieties in partial flag varieties $G/P$ for a connected semisimple algebraic group $G$. 
Using Deodhar's decomposition of Richardson varieties and the work of Pasquier, 
we give an explicit description of the fan of a toric Schubert variety, leading to a combinatorial model for its cones. 
As an application, we obtain necessary and sufficient conditions for smoothness of toric Schubert varieties in terms of the Cartan integers associated to a reduced expression. 
Furthermore, we prove that for a Coxeter-type element $w \in W^P$, the interval $[e,w]_{W^P}$ is a supersolvable join-distributive lattice. Finally, we apply these results to the study of spherical and horospherical Schubert varieties, providing a combinatorial method for checking the smoothness via the associated toric Schubert varieties.
\end{abstract}



\section{Introduction}

A normal irreducible variety $X$ equipped with an algebraic action of a torus $T$ is called \emph{toric} if $X$ contains a dense open $T$-orbit. 
More generally, a normal irreducible variety $X$ with an action of a semisimple algebraic group $G$ is called \emph{spherical} if it contains a dense open orbit for a Borel subgroup $B \subset G$. 
Toric varieties form a distinguished subclass of spherical varieties, and their rich combinatorial structure makes them a natural testing ground for geometric questions~\cite{Fulton1993}.
For related recent work on Levi-spherical Schubert varieties, see  ~\cite{CS2023} and ~\cite{GaoHodgesYong}. 
More specifically, our goal is to characterize toric Schubert varieties in partial flag varieties and to analyze their geometry using combinatorial and representation-theoretic tools.

\medskip

Let $G$ be a connected semisimple algebraic group, $B \subset G$ a Borel subgroup, and $T \subset B$ a maximal torus. 
For $w$ in the Weyl group $W$, let $X_{wB} \subset G/B$ denote the corresponding Schubert variety. 
It is known that $X_{wB}$ is a toric variety for the action of $T$ if and only if $w$ is of \emph{Coxeter type}, that is, $w$ can be written as a product of pairwise distinct simple reflections (see \cite{Karuppuchamy2013,LeeMasudaPark2023}). 
This result establishes a direct link between toric geometry and the combinatorics of reduced expressions.
One of the main objectives of our paper is to extend this characterization to Schubert varieties in arbitrary partial flag varieties $G/P$. 
Our approach is based on a detailed analysis of Deodhar's decomposition of Richardson varieties (studied in \cite{CS2025, Deodhar1985, MarshRietsch2004}; 
see also \cite{GorskyKimShermanBennett2026} for recent results on toric Richardson varieties).
In particular, we show that the dense Deodhar component admits a natural torus parametrization, and that the toric property can be detected from the combinatorics of reduced expressions. 
This yields a conceptual and uniform proof of the Coxeter-type criterion.

\medskip

We also obtain an explicit description of the fan associated to a toric Schubert variety. 
Using Bott-Samelson coordinates together with a detailed analysis of transition functions, we describe the fan in terms of generators determined by simple roots and Cartan integers. 
This leads to a concrete combinatorial model for the maximal cones, indexed by Weyl group elements, and allows us to translate geometric properties into explicit linear-algebraic conditions. 
For related descriptions in the full flag case, see also \cite{GrossbergKarshon1994, LeeMasudaPark2023, lee2024torusorbitclosuresflag,   KimLee2026}.

\medskip

We then pass to partial flag varieties. 
Let $P \subset G$ be a parabolic subgroup and $W^P$ the set of minimal coset representatives. 
We prove that a Schubert variety $X_{wP} \subset G/P$ is toric for the action of $T$ if and only if $w \in W^P$ is of Coxeter type. 
Moreover, we describe the fan of $X_{wP}$ in terms of the fan of $X_{wB}$ via the natural projection $G/B \to G/P$. 
In type $A$ Grassmannians, related classification and fan descriptions were obtained by Kim and Lee in \cite{KimLee2026}. 
Our main result gives a necessary and sufficient condition for smoothness of toric Schubert varieties in $G/P$. 
More precisely, for a Coxeter-type element $w\in W^P$, we give a combinatorial criterion in terms of the Cartan integers associated to a reduced expression of $w$ (see Theorem~\ref{thm:direct-smoothness}).
Our approach simultaneously explains the toric property, the fan structure, and smoothness of Schubert varieties in both full and partial flag varieties.
From a combinatorial perspective, the Bruhat intervals associated to toric Schubert varieties possess a highly rigid and elegant structure. 
While the lower Bruhat interval of a Coxeter-type element in the full Weyl group is known to be a Boolean lattice, its projection into the quotient $W^P$ is generally not Boolean. 
Nevertheless, we establish that for any Coxeter-type element $w \in W^P$, the interval $[e,w]_{W^P}$ forms the feasible-set lattice of a supersolvable antimatroid in the sense of~\cite{Armstrong2009}. 
Consequently, these intervals are supersolvable join-distributive lattices, revealing a deep connection between the geometry of partial flag varieties and the theory of antimatroids.
\medskip

A central theme of the paper is the study of smoothness of toric Schubert varieties. 
Using our explicit description of the fan, we derive necessary and sufficient conditions for smoothness in terms of Cartan data. 
These conditions admit a combinatorial interpretation in terms of the reduced expression of $w$, leading to a transparent criterion for smoothness of $X_{wP}$.

\medskip

Our results are directly applicable to the \emph{horospherical Schubert varieties}. 
To explain, let us first define the horospherical varieties. 
A closed subgroup $H \subset G$ is called \emph{horospherical} if it contains a maximal unipotent subgroup of $G$, and a normal irreducible $G$-variety is horospherical if the stabilizer of a point in general position is horospherical. 
Horospherical varieties form an important subclass of spherical varieties, and include toric varieties as a special case.

In our previous work \cite{CKS}, we obtained a complete characterization of horospherical Schubert varieties and proved that every such variety is smooth.
More precisely, a Schubert variety $X_{wB}$ is horospherical (for a suitable Levi subgroup) if and only if $w$ admits a factorization
\[
w = w_{0,J} \, c
\]
where $c$ is a Coxeter-type element with support disjoint from $J$. 
This factorization also implies that the corresponding Schubert variety is smooth.
In the present paper, we recover this smoothness phenomenon from a toric viewpoint. 
In particular, our smoothness criteria imply, as a special case, that horospherical Schubert varieties are smooth. 

\medskip

The structure of our paper is as follows.
In Section~\ref{sec:preliminaries}, we fix notation and recall basic facts about Weyl groups, Schubert varieties, and Deodhar decompositions. 
In Section~\ref{sec:fans}, we study toric Schubert varieties in $G/B$, prove the Coxeter-type criterion.
In Section~\ref{sec:Pasquier}, we explicitly describe the fans associated to toric Schubert varieties in $G/B$ and extend these results to partial flag varieties and relate the fan of $X_{wP}$ to that of $X_{wB}$. 
In Section~\ref{sec:intervals}, we prove the expected Bruhat interval characterization for toric Schubert varieties and establish that these projected intervals are supersolvable join-distributive lattices.
In Section~\ref{sec:smooth}, we establish our smoothness criteria for toric Schubert varieties and give several combinatorial reformulations. 
Finally, in Section~\ref{sec:applications}, we apply these results to obtain consequences for horospherical Schubert varieties and related classes of spherical varieties.

\section{Notation and Preliminaries}
\label{sec:preliminaries} 

In this section, we fix notation and recall basic facts on algebraic groups,
root systems, Schubert varieties, and Deodhar decomposition.
Standard references include  \cite{BilleyLakshmibai}, \cite{Brion2005}, \cite{BrionKumar}, \cite{Hartshorne}, \cite{Springer}, \cite{Humphreys}, and \cite{Humphreys1972}.
For Coxeter groups, Bruhat order, reduced expressions, and the subword criterion, we also use
\cite{BjornerBrenti, Humphreys1990}.
For toric varieties and fans, we use the conventions of
\cite{CLS, Fulton1993}.

Let $G$ be a semisimple algebraic group over an algebraically closed field $k$ of characteristic 0.
Fix a maximal torus $T \subset G$, and denote by $W := N_G(T)/T$ the Weyl group of $G$ with respect to $(G, T)$.
Let $R$ be the root system of $G$ with respect to $(G,T)$.

Fix a Borel subgroup $B \subset G$ containing $T$, and let
\[
B^- := n_0 B n_0^{-1}
\]
be the opposite Borel subgroup, where $n_0 \in N_G(T)$ represents the longest element $w_0 \in W$.

Let $R^{+}$ (resp. $R^{-}$) be the set of positive (resp. negative) roots with respect to $B$, so that $R^{-} = -R^{+}$.
Let
\[
S = \{\alpha_1, \dots, \alpha_n\} \subset R^{+}
\]
be the set of simple roots, where $n = \mathrm{rank}(G)$.
For each $\alpha_i \in S$, we denote by $s_i \in W$ the corresponding simple reflection.

An element $w \in W$ is said to be of \emph{Coxeter type} if it can be written as a product of pairwise distinct simple reflections.

Let $X(T)$ and $Y(T)$ denote the character and cocharacter lattices of $T$, respectively.
These are free abelian groups of the same rank, equipped with a perfect pairing
\[
\langle \cdot , \cdot \rangle : Y(T) \times X(T) \longrightarrow \mathbb{Z}
\]
defined by the exponent of the natural evaluation map.

Set
\[
Y(T)_\R := Y(T) \otimes_\Z \R, \qquad X(T)_\R := X(T) \otimes_\Z \R,
\]
and extend the pairing $\langle \cdot, \cdot \rangle$ to a real-valued bilinear form.
Let $(\cdot,\cdot)$ be a $W$-invariant positive definite symmetric bilinear form on $X(T)_\R$.
For $\alpha \in R$, define the associated reflection by
\[
s_\alpha(\lambda) = \lambda - \frac{2(\lambda,\alpha)}{(\alpha,\alpha)} \alpha,{\text{~where~} \lambda \in X(T)_{\R}.}
\]
The coroot $\alpha^\vee \in Y(T)_\R$ corresponding to the simple root $\alpha$ is defined by
\[
\langle \alpha^\vee, \lambda  \rangle = \frac{2(\lambda,\alpha)}{(\alpha,\alpha)}, {\text{~where~} \lambda \in X(T)_{\R}.}
\]
A weight $\lambda \in X(T)$ is called \emph{dominant} if $\langle \alpha_i^\vee, \lambda \rangle \ge 0$ for all $\alpha_i \in S$.
The set of dominant weights is denoted by $X(T)_+$.
For each $\alpha_i \in S$, let $\varpi_i$ denote the corresponding fundamental weight,
characterized by
\[
\langle  \alpha_j^\vee, \varpi_i \rangle = \delta_{ij},
\]
where $\delta_{ij}$ denotes the Kronecker delta.
\medskip

Let $P \supset B$ be a parabolic subgroup. 
Define
\[
S_P := \{ \alpha \in S \mid U_\alpha \subset P \},
\]
and let $W_P \subset W$ be the subgroup generated by $\{ s_\alpha : \alpha \in S_P \}$.
The \emph{partial flag variety} associated to $P$ is $G/P$.

For each $w \in W^P$ (the set of minimal length representatives of the right cosets in $W/W_P$),
we define the \emph{Schubert} and \emph{opposite Schubert} varieties in $G/P$ by
\[
X_{wP} := \overline{BwP}/P, \qquad X^{wP} := \overline{B^-wP}/P.
\]

\medskip

For $w \in W$, the \emph{length} of $w$, denoted by $\ell(w)$, is the minimum number of simple reflections required to write $w$ as a product. Thus we obtain the length function
\[
\ell : W \to \mathbb{Z}_{\ge 0}, \quad w \mapsto \ell(w),
\]
which coincides with the dimension function
\[
w \mapsto \dim X_{wB}, \quad w \in W.
\]
\medskip

The $T$-fixed points in $G/P$ will be denoted by $\tau_w$, where $w\in W^P$. We denote the identity coset representative in $W^P$ by $1$ and the corresponding $T$-fixed point in $G/P$ will be denoted by $\tau_1$. 
Finally, $\leqslant$ will denote the Bruhat-Chevalley order on $W^P$. 
For $v\leqslant w$ and $v\neq w$, we write $v < w$.

\subsection{Billey-Postnikov decomposition}

For a parabolic subgroup $P \supset B$ corresponding to a subset of simple roots $S_P \subset S$, we denote the longest element of the parabolic Weyl group $W_P$ by $w_{0,P}$ (or $w_{0,J}$ if the parabolic is indexed by a subset $J \subset S$). 
Every element $w \in W$ admits a unique parabolic factorization $w = u v$, where $u \in W^P$ is a minimal length representative of a right coset in $W/W_P$, and $v \in W_P$. 
This factorization is length-additive, meaning $\ell(w) = \ell(u) + \ell(v)$. 

While $W^P$ denotes the set of minimal length representatives of the right cosets in $W/W_P$, we similarly define ${}^P W$ as the set of minimal length representatives of the left cosets in $W_P \backslash W$. 
Taking the inverse of an element naturally swaps these sets: $w \in W^P$ if and only if $w^{-1} \in {}^P W$. 
Furthermore, if $J \subseteq K \subseteq S$ are subsets of simple roots, we will frequently utilize the intersection of these sets of representatives, denoted by $W_K^J := W^J \cap W_K$. 
This set represents the minimal length representatives of $W_K / W_J$.

The concept of a \emph{Billey-Postnikov (BP) decomposition} originates from the study of Poincar\'e polynomials of rationally smooth Schubert varieties. 
Originally, in~\cite{BilleyPostnikov}, Billey and Postnikov showed that if a finite-type Schubert variety in $G/B$ is rationally smooth, the Weyl group element $w$ (or its inverse) admits a parabolic decomposition $w = vu$ with respect to a subset $K \subset S$ (where $S \setminus K$ is a single leaf of the Dynkin diagram) such that its Poincar\'e polynomial factors as $P_w(t) = P_v^K(t) P_u(t)$. 
Later, Richmond and Slofstra generalized this concept to arbitrary partial flag varieties by dropping the restrictive leaf condition. 
Following their terminology in~\cite{RichmondSlofstra}, if $J \subseteq K \subseteq S$ are subsets of simple roots and $w \in W^J$, a parabolic decomposition $w = vu$ (where $v \in W^K$ and $u \in W_K^J$) is defined as a BP decomposition with respect to $(J, K)$ if the relative Poincar\'e polynomials strictly satisfy $P_w^J(t) = P_v^K(t) P_u^J(t)$. 
Crucially for geometric applications, Richmond and Slofstra established that this polynomial factorization is combinatorially equivalent to $u$ being the unique maximal element in the Bruhat intersection $[e, w] \cap W_K^J$. This equivalence provides the exact condition required for the natural projection between Schubert varieties to be a Zariski-locally trivial fiber bundle.

\subsection{Deodhar components, subexpressions of a reduced expression}

This subsection provides us with the main technical tool that is due to Deodhar~\cite{Deodhar1985}.
We will follow the exposition of Marsh and Rietsch~\cite{MarshRietsch2004}, while keeping track of all intermediate constructions.

\medskip

Let $w\in W$.
An \emph{expression} for $w$, denoted by $\mathbf{w}$, is a sequence $(w_{(0)},\dots,w_{(r)})\in W^{r+1}$ such that 
\begin{enumerate}
\item $w_{(0)}=1$,
\item $w_{(r)}=w$, 
\item $w_{(j)}$ is either $w_{(j-1)}$ or $w_{(j-1)}s_i$ for some $s_i\in S$. 
\end{enumerate}
Since for every $j\in \{1,\dots, r\}$, the product $w_{(j-1)}^{-1}w_{(j)}$ is an element of $S\cup \{1\}$,
the data of the expression $\mathbf{w}$ is equivalent to its \emph{sequence of factors}
\[
(w_{(1)}, w_{(1)}^{-1}w_{(2)}, \dots, w_{(r-1)}^{-1}w_{(r)}).
\]

\medskip

\begin{example}
\label{ex:runningexample}
Let $G=SL_3$ so that $W \cong S_3$ is generated by $s_1, s_2$.
Let $w = s_1 s_2 s_1$.
A reduced expression is given by
\[
\mathbf{w} = (1, s_1, s_1 s_2, s_1 s_2 s_1).
\]
The corresponding sequence of factors is $(s_1, s_2, s_1)$.
\end{example}

Let $\mathbf{w} = (w_{(0)},\dots, w_{(r)})$ be an expression of length at least $2$.
We define three subsets of $\{1,\dots,r\}$:
\begin{align*}
J_{\mathbf{w}}^{+} & :=\{j: w_{(j-1)}<w_{(j)}\}, \\
J_{\mathbf{w}}^{\circ} &:= \{j: w_{(j-1)}=w_{(j)}\}, \\
J_{\mathbf{w}}^{-} &:= \{j: w_{(j)}<w_{(j-1)}\}.
\end{align*}
We say that $\mathbf{w}$ is \emph{nondecreasing} if $J_{\mathbf{w}}^{-}=\emptyset$,
and \emph{reduced} if $J_{\mathbf{w}}^- \cup J_{\mathbf{w}}^\circ = \emptyset$.

\medskip

\begin{example}
For the expression $\mathbf{w}$ from Example~\ref{ex:runningexample}, we have
\[
J_{\mathbf{w}}^+ = \{1,2,3\}, \qquad J_{\mathbf{w}}^\circ = \emptyset, \qquad J_{\mathbf{w}}^- = \emptyset.
\]
Hence, $\mathbf{w} = (1, s_1, s_1 s_2, s_1 s_2 s_1)$ is reduced. 
\end{example}

\medskip

Let $(s_{i_1},\dots,s_{i_r})$ be the sequence of factors of a reduced expression $\mathbf{w}$ for $w\in W$,
and let $v\in W$ with $v\leqslant w$.
A \emph{subexpression for $v$ in $\mathbf{w}$} is a sequence
\[
\mathbf{v} = (v_{(0)},\dots,v_{(r)})
\]
such that
\[
v_{(j)} \in \{v_{(j-1)},\, v_{(j-1)} s_{i_j}\}.
\]
If, in addition,
\[
v_{(j)} \leqslant v_{(j-1)} s_{i_j}
\]
holds for all $j$, then $\mathbf{v}$ is called a \emph{distinguished subexpression}.
This condition ensures that whenever right multiplication by $s_{i_j}$ decreases Bruhat order, one must take the step $v_{(j)} = v_{(j-1)} s_{i_j}$.
When $\mathbf{v}$ is a distinguished subexpression of $\mathbf{w}$, we will write $\mathbf{v} \preceq \mathbf{w}$.

\medskip

\begin{example}
We continue with Example~\ref{ex:runningexample}.
Let $w = s_1 s_2 s_1$ and $v = s_1$.
A subexpression is obtained by choosing which reflections to apply:
\[
\mathbf{v} = (1, s_1, s_1, s_1).
\]
This is distinguished since all steps satisfy the required inequalities.
\end{example}

\medskip

A distinguished subexpression $\mathbf{v}$ is called \emph{positive} if
\begin{equation}\label{positive-condition}
v_{(j-1)} < v_{(j-1)} s_{i_j} \qquad \text{for all $j$.}
\end{equation}
Equivalently, this means that right multiplication by $s_{i_j}$ increases the Bruhat order at every step.
In particular, the conditions
\[
v_{(j)} \in \{v_{(j-1)},\, v_{(j-1)} s_{i_j}\}
\quad \text{and} \quad
v_{(j)} \leqslant v_{(j-1)} s_{i_j}
\]
continue to hold, but the positivity condition imposes the additional requirement
that $v_{(j-1)} < v_{(j-1)} s_{i_j}$.
We denote a positive subexpression by $\mathbf{v}_+$.

A fundamental result (see~\cite[Lemma 3.5]{MarshRietsch2004}) states that for every $v\leqslant w$,
there exists a unique positive subexpression $\mathbf{v}_+$ of $v$ in $\mathbf{w}$.

\begin{example}
We continue with the running example Example~\ref{ex:runningexample}.
Let $w=s_1s_2s_1$ and fix the reduced expression
\[
\mathbf{w}=(1,s_1,s_1s_2,s_1s_2s_1).
\]
Then the sequence of factors is $(s_1,s_2,s_1)$.
We construct the positive subexpression of $v:=s_1$ in $\mathbf{w}$.
Set
\[
\mathbf{v}=(v_{(0)},v_{(1)},v_{(2)},v_{(3)})=(1,1,1,s_1).
\]
We verify first that $\mathbf{v}_+$ is a subexpression of $\mathbf{w}$:
\begin{align*}
v_{(1)}&=1\in\{1,1\cdot s_1\}=\{1,s_1\},\\
v_{(2)}&=1\in\{1,1\cdot s_2\}=\{1,s_2\},\\
v_{(3)}&=s_1\in\{1,1\cdot s_1\}=\{1,s_1\}.
\end{align*}
Next we check that it is distinguished:
\begin{align*}
v_{(1)}&=1\leqslant s_1=v_{(0)}s_{i_1},\\
v_{(2)}&=1\leqslant s_2=v_{(1)}s_{i_2},\\
v_{(3)}&=s_1\leqslant s_1=v_{(2)}s_{i_3}.
\end{align*}
Thus $\mathbf{v}\preceq \mathbf{w}$.
Finally, we check positivity. Since
\[
v_{(0)}=1,\qquad v_{(1)}=1,\qquad v_{(2)}=1,
\]
we have
\begin{align*}
v_{(0)}s_{i_1}&=1\cdot s_1=s_1,\\
v_{(1)}s_{i_2}&=1\cdot s_2=s_2,\\
v_{(2)}s_{i_3}&=1\cdot s_1=s_1.
\end{align*}
Therefore
\[
1<s_1,\qquad 1<s_2,\qquad 1<s_1,
\]
so that
\[
v_{(j-1)}<v_{(j-1)}s_{i_j}\qquad\text{for all }j\in\{1,2,3\}.
\]
Hence $\mathbf{v}$ is positive.
We conclude that the unique positive subexpression of $v=s_1$ in
\[
\mathbf{w}=(1,s_1,s_1s_2,s_1s_2s_1)
\]
is
\[
\mathbf{v}_+=\mathbf{v}= (1,1,1,s_1).
\]
\end{example}

\subsection{Relation to Richardson varieties.}
Let $X_w^v$ be a Richardson variety.
Let $\mathbf{v}_+$ be the positive subexpression corresponding to $v$ in $\mathbf{w}$.
By discarding the indices in $J_{\mathbf{v}_+}^{\circ}$, one obtains a reduced expression for $v$.
In particular,
\[
\ell(v) = \ell(w) - |J_{\mathbf{v}_+}^{\circ}|.
\]
Equivalently,
\[
\dim(X_w^v) = |J_{\mathbf{v}_+}^{\circ}|.
\]
We proceed to define the \emph{Deodhar decomposition}.
Let
\[
\mc{R}_{v,w} := B\tau_w \cap B^- \tau_v.
\]
Since $B=UT$ and $B^-=U^-T$, and since $\tau_w$ and $\tau_v$ are torus fixed points, we have
$\mc{R}_{v,w} = U \tau_w \cap U^- \tau_v$.

Let $\mathbf{w} = (w_{(0)},\dots,w_{(r)})$ be a reduced expression for $w$.
For each $j$, we have a natural projection morphism
\[
\pi^w_{w(j)} : B \tau_w \to B \tau_{w(j)}, \qquad b \tau_w \mapsto b \tau_{w(j)}.
\]
For $x\in \mc{R}_{v,w}$, define
\[
x_j := \pi^w_{w(j)}(x).
\]
Each $x_j$ lies in a unique Bruhat cell $B^- n_u B/B$.
We denote this element $u$ by $v_{(j)}$ and obtain a sequence
\[
\mathbf{v} = (v_{(0)},\dots,v_{(r)}).
\]
Following~\cite[\S 4.1]{MarshRietsch2004}, we define
\[
\mc{R}_{\mathbf{v},\mathbf{w}} := \{x\in \mc{R}_{v,w} : x_j \in \mc{R}_{v_{(j)}, w_{(j)}} \text{ for all } j\}.
\]
These subsets are called \emph{Deodhar components}.
We now list several important properties of the Deodhar components. 
\medskip

\begin{enumerate}
\item $\mathbf{v}\preceq \mathbf{w}$ if and only if $\mc{R}_{\mathbf{v},\mathbf{w}}$ is non-empty,
\item if $\mathbf{v}\preceq \mathbf{w}$, then
\[
\mc{R}_{\mathbf{v},\mathbf{w}} \cong (k^*)^{|J_{\mathbf{v}}^\circ|} \times k^{|J_{\mathbf{v}}^-|}.
\]
\item if $v\leqslant w$, then 
\[
\mc{R}_{v,w} = \bigsqcup_{\mathbf{v}\preceq \mathbf{w}} \mc{R}_{\mathbf{v},\mathbf{w}}.
\]
\item each $\mc{R}_{\mathbf{v},\mathbf{w}}$ is $T$-stable;
\item the component $\mc{R}_{\mathbf{v}_+,\mathbf{w}}$ is dense in $\mc{R}_{v,w}$.
\end{enumerate}
The first three items follow directly from~\cite[Theorem 1.1 and Corollary 1.2]{Deodhar1985}. 
The fourth statement follows from the $T$-equivariance of the projection maps.
The fifth follows from the fact that $|J_{\mathbf{v}_+}^\circ| = \ell(w) - \ell(v)$.

\medskip

We close our preliminaries section by a summary of the notation that we use throughout the paper. 	
\medskip

\begin{center} 
\begin{tabular}{ll} 
\hline 
Symbol & Meaning \\ 
\hline $G$ & Semisimple algebraic group \\ 
$T$ & Maximal torus \\ 
$W$ & Weyl group \\ 
$R$ & Root system \\ 
$R^+, R^-$ & Positive / negative roots \\ 
$S$ & Simple roots \\ 
$s_i$ & Simple reflections \\ 
$B, B^-$ & Borel subgroup corresponding to $R^+$ / $R^-$\\ 
$P(\supset B)$ & Parabolic subgroup \\ 
$W_P$ & Weyl subgroup \\ 
$W^P$ & Minimal representatives of $W/W_P$ \\ 
$X_{wP}$ & Schubert variety \\ 
$X^{wP}$ & Opposite Schubert variety \\ 
$\tau_w$ & $T$-fixed point corresponding to $w\in W^P$\\
$\mathbf{w}$ & Expression for $w$ \\
$\mathbf{v}$ & Subexpression of $\mathbf{w}$ \\
$\mathbf{v} \preceq \mathbf{w}$ & Distinguished subexpression \\
$\mathbf{v}_+$ & Positive subexpression \\
$J_{\mathbf{w}}^+, J_{\mathbf{w}}^\circ, J_{\mathbf{w}}^-$ & Index sets of an expression \\
$\mc{R}_{v,w}$ & $U^- \tau_w \cap U \tau_v$ \\
$\mc{R}_{\mathbf{v},\mathbf{w}}$ & Deodhar component \\
$\pi^w_{w(j)}$ & Projection morphisms along $\mathbf{w}$ \\
\hline 
\end{tabular} 
\end{center}

\section{Fan associated to toric Schubert varieties in a full flag variety}
\label{sec:fans}

Recall that an element $w \in W$ is called a \emph{Coxeter-type element} if it is a product of distinct simple
reflections (not necessarily all), each appearing at most once, in some order.
We say that an element $c \in W$ is called a \emph{Coxeter element} if it can be written as a product of all simple reflections, each appearing exactly once, in some order.

It is well-known that a Schubert variety $X_{wB}$ in $G/B$ is toric for the action of a maximal torus $T$ of $G$ if and only if $w$ is a Coxeter-type element; see Karuppuchamy~\cite{Karuppuchamy2013} and also \cite{LeeMasudaPark2023}. 
We recovered Karuppuchamy's theorem in~\cite{CS2025} using Deodhar's decomposition of the intersection of double cosets. 
Since it is essential for the subsequent arguments of the current paper, we briefly recall that argument.

\begin{lemma}
\label{lem:toricSchubertinG/B}
The Schubert variety $X_{wB}$ is a toric $T$-variety if and only if $w$ is a Coxeter-type element.
\end{lemma}

\begin{proof}
Let $\mathbf{w}$ be a reduced expression for $w\in W$ with sequence of factors $(s_{i_1},s_{i_2},\ldots, s_{i_r})$. 
Let $\mathbf{1}_{+}=(1,\ldots, 1)$ denote the unique positive subexpression for $1$ in $\mathbf{w}$.
Then $J_{\mathbf{1}_+}^{\circ} = \{1,\ldots, r\}$.
We know that $\mc{R}_{1,w}$ is an open subset of $X_w$.
Hence, $\mc{R}_{1,w}$ is a toric variety if and only if the dense stratum $\mc{R}_{\mathbf{1}_{+}, \mathbf{w}}$ is a toric variety. 
Therefore, it suffices to analyze $\mc{R}_{\mathbf{1}_{+}, \mathbf{w}}$.

Following~\cite[\S 5.1]{MarshRietsch2004}, for $\mathbf{v}\preceq \mathbf{w}$ we define
\begin{equation*}
G_{\mathbf{v},\mathbf{w}} :=
\left\{
g=g_1g_2\cdots g_r \ \Bigg|\ 
\begin{matrix}
g_j = x_{i_j} (t_j) s_{i_j} &\text{if $j\in J^-_{\mathbf{v}}$} \\ 
g_j = y_{i_j} (a_j) &\text{if $j\in J^\circ_{\mathbf{v}}$} \\
g_j = s_{i_j} &\text{if $j\in J^+_{\mathbf{v}}$}
\end{matrix}
\right\},
\end{equation*}
where $a_j \in k^*$ and $t_j\in k$.
Here, $x_{i_j}$ (resp. $y_{i_j}$) denotes the root subgroup corresponding to $\alpha_{i_j}$ (resp. $-\alpha_{i_j}$).
Recall from the preliminaries section that, there is an isomorphism
$(k^*)^{ | J_{\mathbf{v}}^\circ|} \times k^{ | J_{\mathbf{v}}^-|}\;\cong\; \mc{R}_{\mathbf{v},\mathbf{w}}$. 
In our case, $\mathbf{v}=\mathbf{1}_+$ satisfies $J_{\mathbf{1}_+}^-=\emptyset$, so we obtain
$\mc{R}_{\mathbf{1}_+,\mathbf{w}} \cong (k^*)^r$. 
More precisely, we have 
\begin{align*}
\mc{R}_{\mathbf{1}_+,\mathbf{w}} 
&= \left\{  \prod_{m=1}^r y_{\alpha_{i_{m}}} (a_m)B \ \middle|\ (a_1,\dots, a_r)\in (k^*)^r \right\}.
\end{align*}
The $T$-action on $\mc{R}_{\mathbf{1}_+,\mathbf{w}}$  is given by conjugation:
\begin{equation}\label{A:theaction}
t\cdot \prod_{m=1}^r y_{\alpha_{i_m}}(a_m)B
=
\prod_{m=1}^r y_{\alpha_{i_m}}(\alpha_{i_m}(t)^{-1} a_m)B.
\end{equation}
This induces the action on $(k^*)^r$:
\[
t\cdot (a_1,\dots,a_r)
=
(\alpha_{i_1}(t)^{-1} a_1,\dots,\alpha_{i_r}(t)^{-1} a_r).
\]

Let $\xi\in \mc{R}_{\mathbf{1}_+,\mathbf{w}}$.
Then
\[
\mathrm{Stab}_T(\xi)=\bigcap_{m=1}^r \ker(\alpha_{i_m}).
\]
Thus,
\[
\dim T\cdot \xi = r
\quad \Longleftrightarrow \quad
\{\alpha_{i_1},\dots,\alpha_{i_r}\} \text{ are linearly independent}.
\]
Since $\dim X_{wB}=r$, this shows that $X_{wB}$ is toric if and only if the roots $\alpha_{i_m}$ are linearly independent.
This happens if and only if the simple reflections in $\mathbf{w}$ are pairwise distinct, that is, if and only if $w$ is a Coxeter-type element.
\end{proof}

We illustrate the torus action of the previous proof on a simple example. 

\begin{example}
Let $G=SL_3$ and $w=s_1s_2$. Then $r=2$ and
\[
\mathbf{w}=(1,s_1,s_1s_2).
\]
The positive subexpression for $1$ is $(1,1,1)$, hence
$\mc{R}_{\mathbf{1}_+,\mathbf{w}} 
= \{ y_{\alpha_1}(a_1)y_{\alpha_2}(a_2)B \mid a_1,a_2\in k^* \}
\cong (k^*)^2$. 
Then the maximal torus $T\subset SL_3$ acts on $\mc{R}_{\mathbf{1}_+,\mathbf{w}}$ by
\[
t\cdot(a_1,a_2)=(\alpha_1(t)^{-1}a_1,\alpha_2(t)^{-1}a_2),
\]
which is the standard diagonal action on $(k^*)^2$.
\end{example}

\section{Pasquier's description of fans of toric Schubert varieties}
\label{sec:Pasquier}

Let $\mathbf{w}$ be a reduced expression for $w\in W$ with sequence of factors $(s_{i_1},s_{i_2},\ldots, s_{i_r})$. The Bott--Samelson variety associated to $\mathbf{w}$ is defined by
\[
Z(\mathbf{w})
:= P_{\alpha_{i_1}} \times^{B} P_{\alpha_{i_2}} \times^{B} \cdots \times^{B} P_{\alpha_{i_r}} / B,
\]
that is,
\[
Z(\mathbf{w})
:= (P_{\alpha_{i_1}} \times \cdots \times P_{\alpha_{i_r}})/B^{r},
\]
where the action of $B^{r}$ on $P_{\alpha_{i_1}} \times \cdots \times P_{\alpha_{i_r}}$ is given by
\[
(p_1,\ldots,p_r)\cdot(b_1,\ldots,b_r)
=
(p_1 b_1,\,
b_1^{-1} p_2 b_2,\,
\ldots,\,
b_{r-1}^{-1} p_r b_r),
\]
for all $p_j \in P_{\alpha_{i_j}}$ and $b_j \in B$.

\medskip

It is well known that if $w$ is of Coxeter type, then for any reduced expression
$\mathbf{w}$ of $w$ the natural $B$-equivariant surjective birational morphism
\[
\pi_{\mathbf{w}} : Z(\mathbf{w}) \longrightarrow X_{wB}, \qquad [p_1,\ldots, p_r]\mapsto p_1\cdots p_rB,
\]
is an isomorphism. Consequently, the Schubert variety $X_{wB}$ is smooth
(see \cite{Fan1998}).

\begin{lemma}
\label{lem:BS-iso-coxeter}
Let $w \in W$ be a Coxeter-type element, and let $\mathbf{w}$ be any reduced expression of $w$.
Then the natural morphism
\[
\pi_{\mathbf{w}} : Z(\mathbf{w}) \longrightarrow X_{wB}
\]
is an isomorphism.
\end{lemma}

\begin{proof}
Let $(s_{i_1},\dots,s_{i_r})$ be the factor sequence of $\mathbf{w}$. 
Since $\mathbf{w}$ is reduced, $w= s_{i_1}\cdots s_{i_r}$, and the morphism $\pi_{\mathbf{w}}$ is birational and projective.
We will show that it is bijective. Since $Z(\mathbf{w})$ is smooth and $X_{wB}$ is normal, it will
then follow from Zariski's Main Theorem that $\pi_{\mathbf{w}}$ is an isomorphism.

We first recall the $T$-fixed points of $Z(\mathbf{w})$. For $\varepsilon=(\varepsilon_1,\dots,\varepsilon_r)\in\{0,1\}^r$,
set
\[
p_\varepsilon := [(q_1,\dots,q_r)]\in Z(\mathbf{w}),
\qquad
q_j=
\begin{cases}
1,&\varepsilon_j=0,\\
\dot s_{i_j},&\varepsilon_j=1,
\end{cases}
\]
where $\dot s_{i_j}\in N_G(T)$ is a representative of $s_{i_j}$.
It is standard that the set of $T$-fixed points of $Z(\mathbf{w})$ is exactly
\[
Z(\mathbf{w})^T=\{p_\varepsilon\mid \varepsilon\in\{0,1\}^r\},
\]
and that
\[
\pi_{\mathbf{w}}(p_\varepsilon)=\tau_{u_\varepsilon},
\qquad
u_\varepsilon:=s_{i_1}^{\varepsilon_1}\cdots s_{i_r}^{\varepsilon_r}\in W.
\]

We claim that for every $u\leqslant w$, there is a unique $\varepsilon\in\{0,1\}^r$ such that
$u=u_\varepsilon$. 
Equivalently, every element of the interval $[e,w]$ is represented by a unique subword of the fixed reduced word $s_{i_1}\cdots s_{i_r}$.
We prove this by induction on $r=\ell(w)$. If $r=0$, there is nothing to prove. Assume $r\ge 1$ and write
\[
w=w's,\qquad w':=s_{i_1}\cdots s_{i_{r-1}},\ s:=s_{i_r}.
\]
Since $w$ is of Coxeter type, the simple reflection $s$ does not appear among
$s_{i_1},\dots,s_{i_{r-1}}$.

We first show that
\[
u< us\qquad\text{for every}\quad u\leqslant w'.
\]
Indeed, if $us<u$, then $s$ is a right descent of $u$, so some reduced expression of $u$ ends in $s$.
By the subword criterion, $\operatorname{supp}(u)$ is contained in $\operatorname{supp}(w')$, hence
$s\notin \operatorname{supp}(u)$. But $us<u$ would force $s\in \operatorname{supp}(u)$, a contradiction.

Now let $x\leqslant w$. 
By the subword criterion, $x$ is represented by a subword of
\[
s_{i_1}\cdots s_{i_{r-1}}s.
\]
If the last letter $s$ is not used, then $x\leqslant w'$. 
If the last letter is used, then $x=us$ for some $u\leqslant w'$. 
Hence
\[
[e,w]=[e,w']\sqcup \{us\mid u\leqslant w'\},
\]
and the union is disjoint because $us>u$ for every $u\leqslant w'$.
By the induction hypothesis, each $u\leqslant w'$ has a unique representing subword of
$s_{i_1}\cdots s_{i_{r-1}}$, and therefore each element of $[e,w]$ has a unique representing subword
of $s_{i_1}\cdots s_{i_r}$. 
This proves the claim.
It follows that $\pi_{\mathbf{w}}$ induces a bijection
\[
Z(\mathbf{w})^T \xrightarrow{\sim} (X_{wB})^T.
\]

We now show that every fiber of $\pi_{\mathbf{w}}$ consists of a single point. Let $v\leqslant w$.
The Bruhat cell $B\tau_v\subset X_{wB}$ is a $B$-orbit, and $\pi_{\mathbf{w}}$ is $B$-equivariant.
Hence for any $x=b\tau_v\in B\tau_v$ we have
\[
\pi_{\mathbf{w}}^{-1}(x)=
b\cdot \pi_{\mathbf{w}}^{-1}(\tau_v).
\]
Thus it is enough to compute the fiber over $\tau_v$.
But $\pi_{\mathbf{w}}^{-1}(\tau_v)$ is $T$-stable, because $\tau_v$ is $T$-fixed and
$\pi_{\mathbf{w}}$ is $T$-equivariant. Since $Z(\mathbf{w})$ is projective, every non-empty
$T$-stable closed subset contains a $T$-fixed point. Therefore, if
$\pi_{\mathbf{w}}^{-1}(\tau_v)$ contained more than one point, it would contain at least two
$T$-fixed points. This is impossible because we already proved that
$\pi_{\mathbf{w}}:Z(\mathbf{w})^T\to (X_{wB})^T$ is bijective. Hence
\[
\pi_{\mathbf{w}}^{-1}(\tau_v)
\]
is a singleton, and therefore every fiber over the cell $B\tau_v$ is a singleton.
Since
\[
X_{wB}=\bigsqcup_{v\leqslant w} B\tau_v,
\]
it follows that $\pi_{\mathbf{w}}$ is bijective.

Finally, a projective bijective morphism is quasi-finite, hence finite. Since $X_{wB}$ is normal and
$\pi_{\mathbf{w}}$ is finite and birational, Zariski's Main Theorem implies that
$\pi_{\mathbf{w}}$ is an isomorphism.
\end{proof}

\medskip

Pasquier~\cite{Pasquier2010} reformulated the Grossberg and Karshon's work~\cite{GrossbergKarshon1994} in the setting of Bott--Samelson varieties. 
In this section, we recall the description of the fan of toric Schubert varieties following~\cite{Pasquier2010}.

\medskip

Let $w$ be a Coxeter-type element in $W$. By Lemma~\ref{lem:toricSchubertinG/B}, the Schubert variety $X_{wB}$ is a toric $T$-variety. In particular, the open Deodhar component $\mc{R}_{\mathbf{1}_+,\mathbf{w}}$ identifies with the open $T$-orbit in $X_{wB}$.
Using the parametrization from the previous section, we have
\[
\mc{R}_{\mathbf{1}_+,\mathbf{w}}
=
\left\{
\prod_{m=1}^r y_{\alpha_{i_m}}(a_m)B \ \middle|\ (a_1,\dots,a_r)\in (k^*)^r
\right\}.
\]
We now consider the following homomorphism of algebraic groups
\[
\varphi: T \longrightarrow (k^*)^r, \qquad
t \longmapsto (\alpha_{i_1}(t)^{-1},\ldots, \alpha_{i_r}(t)^{-1}).
\]
Since the roots $\alpha_{i_1},\dots,\alpha_{i_r}$ are linearly independent, the map $\varphi$ is surjective. 
Hence the $T$-action on $X_{wB}$ factors through an effective action of the quotient torus
\[
T / \bigcap_{j=1}^r \ker(\alpha_{i_j}) \;\cong\; (k^*)^r.
\]
Thus the character lattice of the acting torus is identified with the lattice generated by
\[
\{-\alpha_{i_1},\dots,-\alpha_{i_r}\} \subset X(T).
\]
For every root $\alpha$, there exists a one-parameter subgroup
$u_{\alpha} : k \longrightarrow U_{\alpha}$
such that
\[
t\,u_{\alpha}(x)\,t^{-1} = u_{\alpha}(\alpha(t)x).
\]
Moreover,
\begin{equation}\label{eq4.2}
n_{\alpha} := u_{\alpha}(1)\,u_{-\alpha}(-1)\,u_{\alpha}(1)
\end{equation}
represents the reflection $s_\alpha$ of $W$ in $N_G(T)$.
We also recall the identity
\begin{equation}\label{eq4.4}
n_{-\alpha}\,u_{-\alpha}(-x)
=
u_{-\alpha}(x^{-1})\,\alpha^{\vee}(x)\,u_{\alpha}(-x^{-1}),
\end{equation}
which will be used below.

\medskip

Let $\mathbf{w}$ be a reduced expression for $w$ with sequence $(s_{i_1},\dots,s_{i_r})$. 
For $1\le m,j\le r$, define
\[
\beta_{i_m, i_j} := \langle \alpha_{i_m}^\vee, \alpha_{i_j} \rangle.
\]

\medskip

\begin{theorem}
\label{generatingmatrix}
Let $w \in W$ be a Coxeter-type element with reduced expression
\[
w = s_{i_1}s_{i_2}\cdots s_{i_r}.
\]
Then the fan of the toric Schubert variety $X_{wB}$ is isomorphic to the fan in $\mathbb{R}^r$ whose primitive ray generators are the columns of the matrix
\[
[e_1^{+}\ \cdots\ e_r^{+}\ e_1^{-}\ \cdots\ e_r^{-}],
\]
where
\[
e_k^+ = \text{standard basis vectors}, \qquad
e_k^- = -e_k^{+} - \sum_{m > k} \beta_{i_k,i_m} e_m^{+}.
\]
For a subset $J \subseteq \{1,\dots,r\}$, the maximal cone corresponding to the fixed point $\tau_v$, where 
$v = \prod_{j\in J} s_{i_j}$, 
is given by
\[
C_v
=
\mathrm{Cone}\big(\{e_k^+ \mid k \notin J\} \cup \{e_k^- \mid k \in J\}\big).
\]
\end{theorem}

\begin{proof}
For each $\varepsilon = (\varepsilon_1,\dots,\varepsilon_r) \in \{0,1\}^r$, define a map
\[
\phi_{\varepsilon} : k^r \longrightarrow X_{wB}
\]
by
\[
\phi_{\varepsilon}(x_1,\ldots,x_r)
=
\prod_{j=1}^r
(n_{-\alpha_{i_j}})^{\varepsilon_j}
u_{-\alpha_{i_j}}\bigl((-1)^{\varepsilon_j}x_j\bigr)B.
\]
Each $\phi_{\varepsilon}(k^r)$ is a $T$-stable affine open subset of $X_{wB}$, and these open subsets cover $X_{wB}$.

\medskip

We proceed to describe the coordinates on the open subset corresponding to $\varepsilon=(0,\dots,0)$.
In this case,
\[
\phi_{(0,\dots,0)}(x_1,\dots,x_r)
=
u_{-\alpha_{i_1}}(x_1)\cdots u_{-\alpha_{i_r}}(x_r)B.
\]
This identifies $\phi_{(0,\dots,0)}(k^r)$ with affine space $k^r$, with coordinate functions
\[
f_i\bigl(u_{-\alpha_{i_1}}(x_1)\cdots u_{-\alpha_{i_r}}(x_r)B\bigr)=x_i.
\]
We now describe how the coordinates change when passing to a general chart $\phi_{\varepsilon}(k^r)$.
Fix $\varepsilon \in \{0,1\}^r$ and write a point in $\phi_{\varepsilon}(k^r)$ as
\[
\prod_{j=1}^r
(n_{-\alpha_{i_j}})^{\varepsilon_j}
u_{-\alpha_{i_j}}\bigl((-1)^{\varepsilon_j}x_j\bigr)B.
\]
Using the relations among the root subgroups and the elements $n_{-\alpha}$ (in particular the identities recalled above), we can successively move all factors $n_{-\alpha_{i_j}}$ to the right. 
This expresses the point in the form
\[
u_{-\alpha_{i_1}}(x_1')\cdots u_{-\alpha_{i_r}}(x_r')B,
\]
where the new coordinates $x_j'$ are given by
\[
x_j' = x_j^{(-1)^{\varepsilon_j}} \prod_{m < j} x_m^{(-1)^{\varepsilon_j}\beta_{i_m, i_j}}.
\]
Thus, in the chart $\phi_{\varepsilon}(k^r)$, the coordinate functions $f_i$ can be expressed in terms of $(x_1,\dots,x_r)$ via these transformations.

Let $\chi_1,\dots,\chi_r$ denote the standard basis of the character lattice of $(k^*)^r$, and let $e_1^+,\dots,e_r^+$ be the dual basis.
Consider a character $\chi=\sum_{i=1}^r k_i \chi_i$ and the corresponding monomial
\[
\prod_{i=1}^r f_i^{k_i}.
\]
We determine when this function is regular on $\phi_{\varepsilon}(k^r)$.

\medskip

If $\varepsilon_i=0$, then $x_i$ appears with nonnegative exponent in the expression of $f_i$, and regularity requires
\[
k_i \ge 0.
\]
If $\varepsilon_i=1$, then $x_i$ appears with exponent $-1$ in $x_i'$, and additional contributions arise from variables $x_j$ with $j>i$ through the transformation formula. A direct inspection shows that regularity is equivalent to
\[
- k_i - \sum_{j>i} \beta_{i_i,i_j} k_j \ge 0.
\]
These inequalities define a rational polyhedral cone in the dual space, whose primitive generators are
\[
\{e_k^+ \mid \varepsilon_k=0\} \cup \{e_k^- \mid \varepsilon_k=1\},
\]
where
\[
e_k^- = -e_k^{+} - \sum_{m > k} \beta_{i_k,i_m} e_m^{+}.
\]

\medskip

Therefore, the cones corresponding to the affine charts $\phi_{\varepsilon}(k^r)$ are precisely those described in the statement, and they form the fan of $X_{wB}$.
\end{proof}

\medskip

\begin{example}
\label{ex:fanofHirzebruch}
Let $G=SL_3$ and $w=s_1s_2$. Then $r=2$, and $\beta_{i_1,i_2}=\langle \alpha_1^\vee,\alpha_2\rangle=-1$.
Hence, we have 
\[
e_1^+=(1,0), \qquad  e_2^+=(0,1),
\]
and
\[
e_1^-=-e_1^{+}-\beta_{i_1,i_2}e_2^{+}=(-1,1), \qquad
e_2^-=(0,-1).
\]
Therefore the fan of $X_{wB}$ has rays
\[
(1,0),\ (0,1),\ (-1,1),\ (0,-1),
\]
with maximal cones
\[
\mathrm{Cone}(e_1^+,e_2^+),\quad
\mathrm{Cone}(e_1^-,e_2^+),\quad
\mathrm{Cone}(e_1^+,e_2^-),\quad
\mathrm{Cone}(e_1^-,e_2^-).
\]
This is the standard fan of the Hirzebruch surface $\mathbb{F}_{1}$ obtained by the blow-up of $\mathbb{P}^2$ at a torus-fixed point.
\end{example}

\subsection{Partial flag varieties}

We now consider the toric Schubert varieties in partial flag varieties $G/P$, where $P \supset B$ is a parabolic subgroup.
Consider the natural projection \[\pi_{P} : G/B \longrightarrow G/P.\] 
This map is $G$-equivariant, in particular, both $B$ (resp. $T$)-equivariant.
Let $w \in W^{P}$. 
The maximal torus $T$ acts on $X_{wP}$ by left multiplication.
The restriction of $\pi_{P}$ to the Schubert cell $BwB/B$ induces a $B$-equivariant isomorphism:
\begin{equation}\label{eq:celliso}
\pi_{P} : BwB/B \xrightarrow{\;\sim\;} BwP/P .
\end{equation}

\medskip

\begin{example}
Let $G=SL_3$ and let $P=P_{\widehat{\alpha_1}}$ be the maximal parabolic subgroup corresponding to the simple root $\alpha_1$. Then $G/P \cong \mathbb{P}^2$.
The Weyl group is $W=S_3$, and $W^P = \{1,s_1,s_2s_1\}$.
For $w=s_2s_1$, the Schubert variety $X_{wP}$ is the whole projective plane $\mathbb{P}^2$, while for $w=s_1$, the Schubert variety $X_{wP}$ is a projective line in $\mathbb{P}^2$.
\end{example}

\medskip
We now record an expected result whose proof is not present in the literature.

\begin{proposition}
\label{prop:toric}
Let $w\in W^{P}$. Then the Schubert variety $X_{w}^{P}\subset G/P$ is toric for the action of $T$ if and only if $w$ is a Coxeter-type element.
\end{proposition}

\begin{proof}
Let
\[
\pi_P:G/B\to G/P
\]
be the natural projection. Since $w\in W^{P}$, the restriction of $\pi_P$ to the Schubert cell is a $B$-equivariant, hence $T$-equivariant, isomorphism
\[
\psi:=\pi_P\big|_{BwB/B}: BwB/B \xrightarrow{\sim} BwP/P.
\]
Both cells have dimension $\ell(w)$.

First assume that $w$ is of Coxeter type. 
By Lemma~\ref{lem:toricSchubertinG/B}, $X_{wB}$ is a toric $T$-variety. 
In particular, the dense Deodhar component
\[
\mathcal{R}_{{\bf 1}^{+},{\bf w}}\subset BwB/B
\]
is the open $T$-orbit in $X_{wB}$. Choose $\xi\in \mathcal{R}_{{\bf 1}^{+},{\bf w}}$, and set
\[
\xi_0:=\psi(\xi)\in BwP/P\subset X_{wP}.
\]
Since $\psi$ is $T$-equivariant, it restricts to an isomorphism
\[
\psi: T\cdot \xi \xrightarrow{\sim} T\cdot \xi_0.
\]
Hence
\[
\dim(T\cdot \xi_0)=\dim(T\cdot \xi)=\ell(w).
\]
On the other hand, $\dim X_{wP}=\ell(w)$. Therefore $T\cdot \xi_0$ is a locally closed irreducible subset of $X_{wP}$ having the same dimension as $X_{wP}$. It follows that $T\cdot \xi_0$ is open in $X_{wP}$, hence dense. 
Thus $X_{wP}$ is toric.

Conversely, assume that $X_{wP}$ is toric for the action of $T$. Let $O_P\subset X_{wP}$ be its dense open $T$-orbit. Since the Schubert cell $BwP/P$ is an open dense subset of $X_{wP}$, the intersection
\[
O_P\cap BwP/P
\]
is non-empty and open in $BwP/P$. 
We fix a point
\[
\xi_0\in O_P\cap BwP/P.
\]
Then $O_P=T\cdot \xi_0$, and since $O_P$ is open dense in $X_{wP}$,
\[
\dim(T\cdot \xi_0)=\dim X_{wP}=\ell(w).
\]
Let
\[
\xi:=\psi^{-1}(\xi_0)\in BwB/B.
\]
Again, since $\psi$ is $T$-equivariant, it induces an isomorphism of $T$-orbits
\[
\psi: T\cdot \xi \xrightarrow{\sim} T\cdot \xi_0.
\]
Therefore
\[
\dim(T\cdot \xi)=\dim(T\cdot \xi_0)=\ell(w).
\]
But $X_{wB}$ is irreducible of dimension $\ell(w)$. 
Hence the orbit $T\cdot \xi\subset X_{wB}$ is a locally closed subset of dimension equal to $\dim X_{wB}$, so it is open in $X_{wB}$. Consequently its closure is all of $X_{wB}$, and $X_{wB}$ is toric for the action of $T$.
By the full flag case lemma \ref{lem:toricSchubertinG/B}, $X_{wB}$ is toric if and only if $w$ is of Coxeter type.  
Therefore $w$ is of Coxeter type.
\end{proof}

\medskip

\begin{proposition}\label{cone}
Let $w \in W^P$ be such that the Schubert variety $X_{wP}$ in $G/P$ is toric for the action of $T$. Then for every $v\in W^P$ satisfying $v \leqslant w$, the corresponding maximal cone $C_{v}^{P}$ in the fan $\Sigma_w^{P}$ of $X_{wP}$ is obtained by merging the full-flag cones in the same $W_P$-coset:
\[
C_{v}^{P}
=
\bigcup_{\substack{u\le w\\ uW_{P}=vW_{P}}} C_{u}^{B}.
\]
Here the union is understood as equality of supports of the merged cone in the coarsened fan.
\end{proposition}

\begin{proof}
Since $X_{wP}$ is toric for the action of $T$, Proposition~\ref{prop:toric} implies that $w$ is a Coxeter-type element. Hence $X_{wB}$ is also a toric $T$-variety.
Let $\Sigma_w^B$ and $\Sigma_w^{P}$ denote the fans of $X_{wB}$ and $X_{wP}$, respectively. Consider the restriction
\[
\pi_{P}|_{X_{wB}} : X_{wB} \longrightarrow X_{wP}.
\]
This is a $T$-equivariant morphism between toric varieties.

\medskip

Let $\xi \in X_{wB}$ lie in the open $T$-orbit, and set $\xi' := \pi_{P}(\xi)$. Then the induced map
\[
T \cdot \xi \longrightarrow T \cdot \xi'
\]
is an isomorphism of tori. Consequently, the induced maps on character and cocharacter lattices are isomorphisms, and we may identify the lattices of the two tori.
Thus $\pi_{P}|_{X_{wB}}$ induces a map of fans
\[
\Sigma_w^B \longrightarrow \Sigma_w^{P}
\]
which preserves the ambient lattice.

For each maximal cone $C_u^B \in \Sigma_w^B$, there exists a unique maximal cone $C_v^{P} \in \Sigma_w^{P}$ such that
\[
C_u^B \subseteq C_v^{P}.
\]
By the orbit--cone correspondence for toric varieties (see \cite[Chapter~3]{CLS}), this inclusion holds if and only if the corresponding $T$-fixed points satisfy
\[
\pi_{P}(uB) = vP.
\]
Therefore, we have 
\[
C_v^{P}
=
\bigcup_{\substack{u \leqslant w \\ \pi_{P}(uB)=vP}} C_u^B.
\]
Since $\pi_{P}(uB)=vP$ if and only if $uW_P = vW_P$, the result follows.
\end{proof}

\medskip

\begin{example}
Let $G=SL_3$ and let $P=P_{\widehat{\alpha_2}}$ be the maximal parabolic subgroup corresponding to the simple root $\alpha_2$.
Then $G/P \cong \mathbb P^2$. Let $w=s_1s_2\in W^{P}$, which is of Coxeter type. By Example~\ref{ex:fanofHirzebruch}
the fan $\Sigma_w^B$ is the fan of the Hirzebruch surface with rays spanned by 
\[
(1,0),\ (0,1),\ (-1,1),\ (0,-1).
\]
The projection $\pi_P$ identifies maximal cones corresponding to Weyl group elements
in the same coset modulo $W_P=\langle s_1\rangle$. Concretely, the two cones in the same
$W_P$-coset meet along the ray $(0,1)$ and are merged; the resulting fan has rays spanned by 
\[
(1,0),\ (0,-1),\ (-1,1),
\]
which is the complete fan of $\mathbb P^2$. Thus $\Sigma_w^P$ is the fan of
$\mathbb P^2$.
\end{example}

\section{Interval Characterization}
\label{sec:intervals}

From a combinatorial point of view, a Coxeter-type element is a word with no repeated
letters. 
In Bruhat order, elements below a Coxeter-type element are obtained uniquely by choosing a subset of the letters in a fixed reduced expression. 
The next proposition identifies toric Schubert varieties in $G/P$ by Boolean intervals in $W$.

\begin{proposition}
\label{prop:toric-boolean}
Let $w\in W^P$. Then the following are equivalent:
\begin{enumerate}
\item $X_{wP}$ is a toric variety for the action of $T$;
\item $w$ is a Coxeter-type element;
\item the lower Bruhat interval $[e,w]_W$ is isomorphic to the Boolean lattice
$B_{\ell(w)}$ of rank $\ell(w)$.
\end{enumerate}
\end{proposition}

\begin{proof}
The equivalence $(1)\Longleftrightarrow(2)$ is exactly Proposition~\ref{prop:toric}.
It remains to prove $(2)\Longleftrightarrow(3)$.

\medskip

\noindent
$(2)\Rightarrow(3)$.
Assume that $w$ is a Coxeter-type element, and fix a reduced expression
\[
w=s_{i_1}s_{i_2}\cdots s_{i_r},
\qquad r=\ell(w),
\]
in which the simple reflections $s_{i_1},\dots,s_{i_r}$ are pairwise distinct.
For each subset $A\subseteq [r]:=\{1,\dots,r\}$, let
$w_A:=\prod_{j\in A} s_{i_j}$, 
where the product is taken in the induced order.
Let $\mathcal{P}([r])$ denote the power set of $[r]$. 
By the subword criterion, every element of $[e,w]_W$ is equal to $w_A$ for some $A\subseteq [r]$.
Thus the map
\[
\phi:\mathcal{P}([r])\longrightarrow [e,w]_W,
\qquad
A\longmapsto w_A
\]
is surjective.
Since the letters $s_{i_1},\dots,s_{i_r}$ are pairwise distinct, the support of $w_A$ is exactly
\[
\operatorname{supp}(w_A)=\{\,s_{i_j}\mid j\in A\,\}.
\]
Hence $A\neq B$ implies $w_A\neq w_B$, so $\phi$ is injective. Therefore $\phi$ is a bijection.

Moreover, if $A\subseteq B$, then $w_A\leq w_B$ by the subword criterion, so $\phi$ is order-preserving.
Conversely, if $w_A\leq w_B$, then again by the subword criterion every reduced expression of $w_A$
appears as a subword of a reduced expression of $w_B$, and therefore
\[
\operatorname{supp}(w_A)\subseteq \operatorname{supp}(w_B).
\]
Since the letters in the chosen reduced word are distinct, this implies $A\subseteq B$.
Thus $\phi$ is a poset isomorphism:
\[
[e,w]_W \cong \mathcal{P}([r]) = B_r = B_{\ell(w)}.
\]

\medskip

\noindent
$(3)\Rightarrow(2)$.
Assume that $[e,w]_W \cong B_{\ell(w)}$, and let $w=s_{i_1}s_{i_2}\cdots s_{i_r}$ 
be any reduced expression for $w$, where $r=\ell(w)$.
For each subset $A\subseteq [r]$, let $w_A$ be the corresponding subword product. By the subword
criterion, every element of $[e,w]_W$ is equal to $w_A$ for some $A$, so the map
\[
\psi:\mathcal{P}([r])\longrightarrow [e,w]_W,
\qquad
A\longmapsto w_A
\]
is surjective. Both sets have cardinality $2^r$, hence $\psi$ is bijective.

If some simple reflection occurs twice in the reduced expression of $w$, say 
$s_{i_p}=s_{i_q}$ for $p<q$, then we get 
\[
\psi(\{p\})=s_{i_p}=s_{i_q}=\psi(\{q\}),
\]
contradicting injectivity. Therefore the simple reflections
$s_{i_1},\dots,s_{i_r}$ are pairwise distinct, so $w$ is a Coxeter-type element.
This finishes the proof of $(2)\Leftrightarrow(3)$, and hence the proposition.
\end{proof}

The previous proposition does not adapt well to the intervals in $W^P$. 	
More precisely, if $w\in W^P$ is a Coxeter-type element, then $[e,w]_{W^P}$ is a projection of the Boolean lattice $[e,w]_W$ into $W^P$ but itself need not be a Boolean lattice.
We illustrate this phenomenon by an example. 

\begin{example}
\label{ex:chainwiththreeelements}
In type $A_2$, let $P=P_{\widehat{\alpha_2}}$ be the maximal parabolic subgroup corresponding to the simple root $\alpha_2$. Let $w=s_1s_2\in W^P$.
Then $[e,w]_{W^P}$ is a chain with three elements, and $X_{wP}\cong \mathbb{P}^2$. 
\end{example}

We proceed to characterize the projected Boolean intervals in $W^P$. 
First, we recall the definition of a supersolvable lattice, a concept introduced by Stanley in~\cite{Stanley1972}.

\begin{definition}
A finite lattice $L$ is called \emph{supersolvable} if there exists a maximal chain
\[
\mathfrak m:\ \hat 0=x_0 \lessdot x_1 \lessdot \cdots \lessdot x_n=\hat 1
\]
such that, for every chain $\mathfrak c$ of $L$, the sublattice generated by
$\mathfrak m \cup \mathfrak c$ is distributive. Such a maximal chain is called an
\emph{$M$-chain}.
\end{definition}

We will also need Armstrong's results from~\cite{Armstrong2009}.

Let $E$ be a finite set equipped with a total order $\preccurlyeq$, and let
$\mathcal{F}\subseteq 2^E$. Following Armstrong~\cite{Armstrong2009}, we say that
$(E,\mathcal{F},\preccurlyeq)$ is a \emph{supersolvable antimatroid} if
$\emptyset\in\mathcal{F}$ and the following condition holds:
\begin{center}
If $A,B\in\mathcal{F}$ with $B\not\subseteq A$, and if
$x=\min_{\preccurlyeq}(B\setminus A)$, then $A\cup\{x\}\in\mathcal{F}$.
\end{center}
The elements of $\mathcal{F}$ are called \emph{feasible sets}. Armstrong proves
that any such family is an antimatroid, and that its feasible-set poset is a
supersolvable join-distributive lattice.
\medskip

Our next result shows that, although the projection of a Boolean Bruhat interval
from $W$ to $W^P$ need not remain Boolean, it still has a very rigid combinatorial
structure.

\begin{theorem}
\label{thm:ssjd-projected-boolean}
Let $w\in W^P$ be a Coxeter-type element. Then the interval $[e,w]_{W^P}$ is
isomorphic to the feasible-set lattice of a supersolvable antimatroid. In particular,
$[e,w]_{W^P}$ is a supersolvable join-distributive lattice, hence a supersolvable lattice.
\end{theorem}

\begin{proof}
Let $w=s_{i_1}s_{i_2}\cdots s_{i_r}$  ($r=\ell(w)$) be a reduced expression for $w$.
For each subset $A\subseteq [r]:=\{1,\dots,r\}$, let
$v_A:=\prod_{j\in A} s_{i_j}$, 
that is, the subword product taken in \emph{increasing order of the indices}.
Since $w$ is Coxeter-type, Proposition~\ref{prop:toric-boolean} identifies the Boolean
lattice $2^{[r]}$ with the Bruhat interval $[e,w]_W$ via
\[
A\longmapsto v_A.
\]
Let
\[
\mathcal{F}:=\{A\subseteq [r]\mid v_A\in W^P\}.
\]
Then the above map restricts to an order-isomorphism
\[
(\mathcal{F},\subseteq)\ \xrightarrow{\ \sim\ }\ [e,w]_{W^P}.
\]
Thus it suffices to show that $\mathcal{F}$ is the family of feasible sets of a
supersolvable antimatroid.

\medskip

First, we claim that
\[
A\in\mathcal{F}
\]
if and only if the following condition holds:
\begin{center}
{for every $j\in A$ such that $\alpha_{i_j}\in S_P$, there exists
$k\in A$ with $k>j$ and $\beta_{i_j,i_k}\neq 0$.}\qquad ($\ast$)
\end{center}
Recall that $v_A\in W^P$ if and only if $v_A$ has no right descent in the set
$\{s_\alpha\mid \alpha\in S_P\}$. Since the simple reflections appearing in the
fixed word $s_{i_1}\cdots s_{i_r}$ are pairwise distinct, every reduced expression
for $v_A$ is obtained from the fixed subword defining $v_A$ by commuting adjacent
commuting generators only: braid moves of length $\ge 3$ cannot occur because they
would require repeated simple reflections. Therefore, for a fixed $j\in A$, the
simple reflection $s_{i_j}$ is a right descent of $v_A$ if and only if it can be
moved to the far right by commuting moves, which is equivalent to saying that it
commutes with every selected letter to its right. Since
\[
s_{i_j}s_{i_k}=s_{i_k}s_{i_j}
\quad\Longleftrightarrow\quad
\beta_{i_j,i_k}=0
\]
for distinct simple reflections, this happens if and only if there is \emph{no}
$k\in A$ with $k>j$ and $\beta_{i_j,i_k}\neq 0$. Hence $v_A\in W^P$ if and only if
condition $(\ast)$ holds, proving the claim.

\medskip

We now discuss the supersolvable antimatroid axiom in our context. 
Let us put the reverse order on the ground set $[r]$:
\[
r\prec r-1\prec \cdots \prec 1.
\]
We will show that $([r],\mathcal{F},\prec)$ is a supersolvable antimatroid.

Let $A,B\in\mathcal{F}$ with $B\not\subseteq A$, and let
$x:=\min_{\prec}(B\setminus A)$. 
By definition of $\prec$, this means that $x$ is the \emph{largest} element of
$B\setminus A$ in the usual order on integers. We must prove that
\[
A\cup\{x\}\in\mathcal{F}.
\]
By our claim from the previous paragraph, it is enough to verify condition $(\ast)$ for $A\cup\{x\}$.

To this end, take $j\in A\cup\{x\}$ with $\alpha_{i_j}\in S_P$.
If $j\in A$, then $A\in\mathcal{F}$, so by Step~1 there exists $k\in A$ with
$k>j$ and $\beta_{i_j,i_k}\neq 0$. Since $A\subseteq A\cup\{x\}$, the same $k$
shows that condition $(\ast)$ still holds for $j$ in $A\cup\{x\}$.
Thus, it remains to consider the case $j=x$. If $\alpha_{i_x}\notin S_P$, there is
nothing to prove. So assume $\alpha_{i_x}\in S_P$. Since $x\in B$ and
$B\in\mathcal{F}$, Step~1 implies that there exists $k\in B$ such that
\[
k>x
\qquad\text{and}\qquad
\beta_{i_x,i_k}\neq 0.
\]
But $x$ is the largest element of $B\setminus A$ in the usual order, so $k\notin B\setminus A$.
Hence $k\in A$. Therefore $k\in A\cup\{x\}$, and again condition $(\ast)$ holds
for $j=x$.
Thus $A\cup\{x\}\in\mathcal{F}$. Since also $\emptyset\in\mathcal{F}$, this proves
that $([r],\mathcal{F},\prec)$ is a supersolvable antimatroid.

Now, since every supersolvable antimatroid is, in particular, an antimatroid, 
we see from the work~\cite{Edelman1980} of Edelman that the feasible-set lattice $(\mathcal{F},\subseteq)$ is
join-distributive (see \cite[Theorem~2.8]{Armstrong2009}).
Moreover, by \cite[Theorem~2.13]{Armstrong2009}, the feasible-set lattice of a
supersolvable antimatroid is supersolvable. 
In other words, $(\mathcal{F},\subseteq)\ \cong\ [e,w]_{W^P}$ is a supersolvable join-distributive lattice. 
This finishes the proof of our theorem.
\end{proof}

\begin{remark}
Example~\ref{ex:chainwiththreeelements} shows that the projected interval need not remain Boolean. 
Of course, that example is a chain, hence still distributive. Thus Theorem~\ref{thm:ssjd-projected-boolean}
should be viewed as a genuine extension of the Boolean case.
Distributivity can nevertheless fail in general. 
For instance, in type $A_3$, let
$P=P_{\alpha_2}$ be the minimal parabolic subgroup corresponding to $\alpha_2$. Let
\[
w=s_2s_1s_3\in W^P.
\]
Then $[e,w]_{W^P}$ is identified with the family
\[
\mathcal F=
\{\emptyset,\{2\},\{3\},\{2,3\},\{1,2\},\{1,3\},\{1,2,3\}\},
\]
ordered by inclusion. 
This lattice is not distributive. 
Indeed, if
\[
x=\{1,2\},\qquad y=\{2\},\qquad z=\{1,3\},
\]
then
\[
x\wedge (y\vee z)=x,
\qquad\text{whereas}\qquad
(x\wedge y)\vee (x\wedge z)=\{2\}.
\]
Hence the conclusion of Theorem~\ref{thm:ssjd-projected-boolean} is genuinely
weaker than distributivity, but much stronger than mere latticehood.
\end{remark}

\section{Smooth Toric Schubert varieties in \texorpdfstring{$G/P$}{G/P}}
\label{sec:smooth}

In this section we characterize smooth toric Schubert varieties in $G/P$.

\begin{lemma}
\label{lem:subsmooth}
Let $X_\Sigma$ be a smooth toric variety. Then every $T$-orbit closure $V(\sigma)$ is smooth.
\end{lemma}
\begin{proof}
The $T$-orbit closures of $X_\Sigma$ are in one-to-one correspondence with cones 
$\sigma\in\Sigma$; the orbit closure corresponding to $\sigma$ is
$V(\sigma):=\overline{O(\sigma)}$.
	Let $\sigma=\langle v_1,\dots,v_k\rangle\in\Sigma$. Since the fan $\Sigma$ is smooth, the primitive generators of any cone can be extended to a $\mathbb{Z}$-basis of the lattice $N$. Extend $v_1,\dots,v_k$ to a basis $v_1,\dots, v_k, v_{k+1},\ldots, v_r$ of $N$
	and let $\tau=\langle v_1,\dots,v_r\rangle$ be a maximal cone containing $\sigma$. 
	
	Since $v_1,\dots,v_r$ is a lattice basis of $N$, let $u_1,\dots,u_r$ be the dual basis of $M$. Then the coordinate functions are $x_i=\chi^{u_i}.$
	The affine toric chart corresponding to $\tau$ satisfies $U_\tau \cong \mathbb{A}^r$ with coordinates $x_1,\dots,x_r$.
	In this chart the orbit closure $V(\sigma)$ is defined by
	\[
	V(\sigma)\cap U_\tau = V(x_1,\dots,x_k)\cong \mathbb{A}^{\,r-k}.
	\]
	Hence $V(\sigma)\cap U_\tau$ is smooth. The open sets $U_\tau$ with $\tau\supset\sigma$ cover $V(\sigma)$, so $V(\sigma)$ is smooth.
\end{proof}

\begin{proposition}
\label{prop:smoothat1P}
Let $w\in W^P$ be a Coxeter-type element with reduced expression
\[
w=s_{i_1}s_{i_2}\cdots s_{i_r}.
\]
If $\alpha_{i_j}\notin S_P$ for all $1\le j\le r$, then the toric Schubert variety $X_{wP}$ is smooth.
\end{proposition}

\begin{proof}
	Since $\alpha_{i_{j}}\notin S_{P}$ for all $1\le j\le r$, the maximal cone $C_{1}^{P}$ corresponding to the identity element $1$ has the extremal ray generators $\{e_1^+, e_2^+, \ldots, e_r^+\}$. 
	Since $e_1^+, e_2^+, \ldots, e_r^+$ spans $\mathbb{Z}^r$, it follows that $C_{1}^{P}$ is a smooth cone. 
	Hence, $X_{wP}$ is smooth at the $T$-fixed point $\tau_{1}:=1P$.

Now the singular locus of $X_{wP}$, denoted $\mathrm{Sing}(X_{wP})$, is a closed $B$-stable subset of $X_{wP}$, since the natural action of $B$ on $X_{wP}$ preserves the singular locus. 
Assume towards a contradiction that $X_{wP}$ is singular. Then $\mathrm{Sing}(X_{wP})$ is a non-empty closed $B$-stable subset of the Schubert variety $X_{wP}$. But $\tau_{1}$ is the unique closed $B$-orbit in $X_{wP}$.
Equivalently, every non-empty closed $B$-stable subset of $X_{wP}$ contains $\tau_{1}$. 
Hence, under our assumption, we have $\tau_{1} \in \mathrm{Sing}(X_{wP})$, contradicting the fact that $X_{wP}$ is smooth at $\tau_{1}$. 
Therefore $\mathrm{Sing}(X_{wP})=\emptyset$, and thus $X_{wP}$ is smooth.
\end{proof}

\begin{remark}
The converse of the previous proposition is false. For example, let
$G=SL_3$, let $P=P_{\widehat{\alpha_2}}$ be the maximal parabolic corresponding to the simple root $\alpha_2$. Let $w=s_1s_2\in W^P$.
Then $\alpha_1\in S_P$, so the hypothesis of the proposition does not hold.
Nevertheless,
\[
X_{wP}=G/P\cong \mathbb{P}^2,
\]
hence $X_{wP}$ is smooth and toric.
Thus the condition $\alpha_{i_j}\notin S_P$ for all $j$ is sufficient but not necessary for smoothness.
\end{remark}

We now have a technical lemma towards the proof of our main result.

\begin{lemma}
\label{lemma:one-variable}
Let $w=s_{i_1}s_{i_2}\cdots s_{i_r}\in W^P$ be a Coxeter-type element, and assume that
\[
\alpha_{i_1}\in S_P
\qquad\text{and}\qquad
\alpha_{i_j}\notin S_P\ \text{ for all } j\ge 2.
\]
Then the toric Schubert variety $X_{wP}$ is smooth if and only if there exists a unique index
\[
2\le l\le r
\]
such that
\[
\beta_{i_1,i_l}=-1,
\qquad
\beta_{i_1,i_j}=0\quad\text{for all } j>1,\ j\ne l.
\]
Equivalently,
\[
e_1^+ + e_1^- = e_l^+,
\]
and the primitive generators of the cone $C_1^P$ are

\[
\{e_i^+\mid i\neq l\}\cup\{e_{1}^-\}.
\]
\end{lemma}

\begin{proof}
Since $w$ is of Coxeter type, Proposition~\ref{prop:toric} shows that $X_{wP}$ is toric.
As explained in the proof of Proposition~\ref{prop:smoothat1P}, $X_{wP}$ is smooth if and only if the cone $C_1^P$ is smooth.

Since $\alpha_{i_1}\in S_P$ and $\alpha_{i_j}\notin S_P$ for $j\ge 2$, Proposition~\ref{cone}
gives
\[
C_1^P = \mathrm{Cone}\bigl(e_1^+,\dots,e_r^+,e_1^-\bigr).
\]
By Theorem~\ref{generatingmatrix},
\[
e_1^- = -e_1^+ - \sum_{m=2}^r \beta_{i_1,i_m} e_m^+,
\]
so
\[
e_1^+ + e_1^- = -\sum_{m=2}^r \beta_{i_1,i_m} e_m^+.
\]
Since $w$ is of Coxeter type, the simple roots $\alpha_{i_1},\dots,\alpha_{i_r}$ are pairwise
distinct. Therefore
\[
\beta_{i_1,i_m}\le 0\qquad\text{for all } m\ge 2.
\]

Assume first that $X_{wP}$ is smooth. Then $C_1^P$ is a smooth cone of dimension $r$, so its
primitive ray generators form a $\Z$-basis of the lattice. The vector $e_1^-$ spans an
extremal ray, because it is the only generator whose coefficient along $e_1^+$ is negative.
Likewise, $e_1^+$ also spans an extremal ray: indeed, if
\[
e_1^+ = a\,e_1^- + \sum_{m=2}^r a_m e_m^+
\qquad\text{with } a,a_m\ge 0,
\]
then comparing the coefficient of $e_1^+$ gives $1=-a$, which is impossible. Hence both
$e_1^+$ and $e_1^-$ must belong to the primitive ray basis of $C_1^P$.

Since $C_1^P$ has $r+1$ generators in a lattice of rank $r$, exactly one of the positive generators
$e_2^+,\dots,e_r^+$ must be redundant. Let that redundant generator be $e_l^+$ with $2\le l\le r$.
Then
\[
e_l^+ = a\,e_1^- + b\,e_1^+ + \sum_{\substack{m=2\\ m\ne l}}^r a_m e_m^+
\qquad\text{for some } a,b,a_m\in\Z_{\ge 0}.
\]
Substituting the expression for $e_1^-$, we obtain
\[
e_l^+
=
(b-a)e_1^+
+\sum_{\substack{m=2\\ m\ne l}}^r \bigl(a_m-a\beta_{i_1,i_m}\bigr)e_m^+
-a\beta_{i_1,i_l}e_l^+.
\]
Comparing coefficients in the basis $e_1^+,\dots,e_r^+$ gives
\[
b=a,
\qquad
1=-a\beta_{i_1,i_l},
\qquad
0=a_m-a\beta_{i_1,i_m}\quad\text{for } m\ne l,\ m\ge 2.
\]
Since $a\ge 0$ and $\beta_{i_1,i_m}\le 0$, the last relation forces
\[
a_m = a\beta_{i_1,i_m}\le 0.
\]
But $a_m\ge 0$, hence $a_m=0$ and $\beta_{i_1,i_m}=0$ for all $m\ne l$, $m\ge 2$.
Then $1=-a\beta_{i_1,i_l}$ implies $a=1$ and $\beta_{i_1,i_l}=-1$.
Thus there is a unique index $l$ such that
\[
\beta_{i_1,i_l}=-1,
\qquad
\beta_{i_1,i_j}=0\quad\text{for all } j>1,\ j\ne l.
\]

Conversely, assume that there exists a unique index $l$ with
\[
\beta_{i_1,i_l}=-1,
\qquad
\beta_{i_1,i_j}=0\quad\text{for all } j>1,\ j\ne l.
\]
Then
\[
e_1^+ + e_1^- = e_l^+,
\]
so $e_l^+$ is redundant. Hence the primitive ray generators of $C_1^P$ are exactly
\[
\{e_1^-,\,e_1^+,\,e_2^+,\dots,\widehat{e_l^+},\dots,e_r^+\}.
\]
Moreover,
\[
e_1^- = -e_1^+ + e_l^+,
\]
so the above set is obtained from the standard basis
\[
\{e_1^+,e_2^+,\dots,e_r^+\}
\]
by replacing $e_l^+$ with $e_1^-$. The corresponding change-of-basis matrix is triangular with
diagonal entries $\pm 1$, hence has determinant $\pm 1$. Therefore the primitive ray generators
form a $\Z$-basis, so $C_1^P$ is smooth. Consequently $X_{wP}$ is smooth.
\end{proof}

\begin{lemma} 
\label{lem:tail-cones}
Let $w=s_{i_1}\cdots s_{i_r}\in W^P$ be a Coxeter-type element, and let
\[
J_P(\mathbf w):=\{\,j\in\{1,\dots,r\}\mid \alpha_{i_j}\in S_P\,\}
=\{j_1<\cdots<j_p\}.
\]
where $p\le r$. 
For $1\le k\le p$, set
\[
w^{(k)}:=s_{i_{j_k}}s_{i_{j_k+1}}\cdots s_{i_r}\in W^P.
\]
Let
\[ N:= \Z e_{1}^+\oplus \Z e_2^+\oplus \cdots \oplus \Z e_r^+, \qquad 
N_k:=\frac{\Z\langle e_1^+, \ldots, e_r^+ \rangle}{\Z\langle e_1^+, \ldots, e_{j_k-1}^+ \rangle}\cong \Z \overline{e_{j_k}^+}\oplus \Z \overline{e_{j_k+1}^+}\oplus \cdots \oplus \Z \overline{e_r^+}
\]
and denote by $\overline{x}$ the image of a vector $x\in N$ in $N_k$.
Then the maximal cone at $\tau_{1}$ of the toric Schubert variety $X_{w^{(k)}P}$ is
$\overline{C}^{(k)}_1 = \mathrm{Cone}\bigl(\overline{e_{j_k}^+},\overline{e_{j_k+1}^+},\dots,\overline{e_r^+},
\overline{d_k},\overline{d_{k+1}},\dots,\overline{d_p}\bigr)$, 
where $d_\ell:=e_{j_\ell}^-$ for $\ell=k,\dots,p$.

In particular, if $X_{wP}$ is smooth, then every $X_{w^{(k)}P}$ is smooth by
Lemma~\ref{lem:subsmooth}, hence every cone $\overline{C}^{(k)}_1$ is smooth.
\end{lemma}

\begin{proof}
The reduced expression for $w^{(k)}$ is the suffix
\[
w^{(k)}=s_{i_{j_k}}s_{i_{j_k+1}}\cdots s_{i_r}.
\]
In this reduced word, the indices belonging to $S_P$ are exactly
\[
j_k<j_{k+1}<\cdots<j_p.
\]
Recall that by Proposition~\ref{cone}, the maximal cone at $\tau_1$ of the toric variety $X_{wP}$ is 
\[
C_1^P = \mathrm{Cone}\bigl(e_1^+,e_2^+,\dots,e_r^+, d_1, d_{2},\dots,d_p\bigr),
\] 
where $d_\ell:=e_{j_\ell}^-$ for $\ell=1,\dots,p$.
The intersections $X_{w^{(k)}P} \cap U_{C_1^P}$ is a $T$-invariant toric subvariety of $U_{C_1^P}$, where $U_{C_1^P}$ is the affine chart  corresponding to $C_1^P$. By the orbit-cone correspondence, it follows that $X_{w^{(k)}P} \cap U_{C^P_1} = \overline{O\!\left(\mathrm{Cone}(e_1^+,e_2^+,\ldots,e_{j_k-1}^+)\right)}.$ 
	The identification of the lattices is given by the quotient map 
\[\pi_k: \Z \langle  e_1^+, \ldots  e_r^+, d_1, d_2, \ldots, d_p \rangle \longrightarrow N_k:= \frac{\Z <e_1^+, \ldots, e_r^+, d_1, d_2, \cdots d_p >}{\Z <e_{1}^+, \ldots, e_{j_k-1}^+>} .\]
Thus, under the quotient map, the cone $\pi_k(C_1^P)=\overline{C}_1^{(k)}$ is precisely the maximal cone corresponding to $\tau_1$ in $X_{w^{(k)}P}\cap U_{C^P_1}$.
The generators of $\overline{C}_1^{(k)}$are precisely
\[
\overline{e_{j_k}^+},\dots,\overline{e_r^+},\overline{d_k},\dots,\overline{d_p}.
\]
This proves the description of $\overline{C}^{(k)}_1$.

If $X_{wP}$ is smooth, then each $X_{w^{(k)}P}$ is a $T$-invariant toric subvariety of $X_{wP}$,
so Lemma~\ref{lem:subsmooth} shows that $X_{w^{(k)}P}$ is smooth. Hence its maximal cone
$\overline{C}^{(k)}_1$ is smooth.
\end{proof}

\begin{lemma} 
\label{lem:successive-replacement}
Let $w=s_{i_1}\cdots s_{i_r}\in W^P$ be a Coxeter-type element, and let
\[
J_P(\mathbf w):=\{\,j\in\{1,\dots,r\}\mid \alpha_{i_j}\in S_P\,\}
=\{j_1<\cdots<j_p\}.
\]
where $p\le r$. 
For $1\le k\le p$, set
\[
w^{(k)}:=s_{i_{j_k}}s_{i_{j_k+1}}\cdots s_{i_r}\in W^P.
\]
If for every $\ell=k+1,\dots,p$ there exists an integer
$t_\ell$ such that
\[
j_\ell<t_\ell\le r,
\qquad
d_\ell=\overline{e_{t_\ell}^+}-\overline{e_{j_\ell}^+}~ \text{~in~} N_{k+1},
\] 
and the integers $t_{k+1},\dots,t_p$ are pairwise distinct, then
\[B_{k+1}:=
\{\overline{e_m^+}\mid j_{k+1}\le m\le r,\ m\notin \{t_{k+1},\dots,t_p\}\}
\cup
\{\overline{d_{k+1}},\dots,\overline{d_p}\}\]
is a $\Z$-basis of $N_{k+1}$, and \[\mathrm{Cone}(\overline{C}^{(k+1)}_1)
=
\mathrm{Cone}\bigl(\{\overline{e_{m}^+}: j_{k+1}\le m\le r, m\notin \{t_{k+1},\ldots, t_{p}\}\}\cup
\{\overline{d_{k+1}},\dots,\overline{d_p}\}\bigr).\]
Moreover, the above relations in $N_{k+1}$ can be lifted in $N$, in particular in $N_k$. 
\end{lemma}

\begin{proof}
We start with the standard basis
\[
\overline{e_{j_{k+1}}^+},\overline{e_{j_{k+1}+1}^+},\dots,\overline{e_r^+}
\]
of $N_{k+1}$ and perform replacements in the order
\[
\ell=p,p-1,\dots,k+1.
\]
At step $\ell$, we replace the basis vector $\overline{e_{t_\ell}^+}$ by
\[
\overline{d_\ell}
=
\overline{e_{t_\ell}^+}-\overline{e_{j_\ell}^+}.
\]
Since $j_\ell<t_\ell$ and because we are proceeding in descending order of $\ell$, the vector
$\overline{e_{j_\ell}^+}$ has not been replaced at the moment when the $\ell$-th step is performed:
indeed, any previously replaced index is of the form $t_m$ with $m>\ell$, and then
\[
t_m>j_m>j_\ell,
\]
so $t_m\ne j_\ell$.
Hence each replacement is an elementary unimodular operation, so after all the replacements
we obtain a $\Z$-basis of $N_{k+1}$. This basis is exactly $B_{k+1}$.
Moreover, at each step the cone generated by the current set of vectors is unchanged, because
\[
\overline{e_{t_\ell}^+}
=
\overline{d_\ell}+\overline{e_{j_\ell}^+}
\]
is a nonnegative linear combination of the new generators. Therefore, after all the replacements,
the generated cone is 
\[\mathrm{Cone}\bigl(\{\overline{e_{m}^+}: j_{k+1}\le m\le r, m\notin \{t_{k+1},\ldots, t_{p}\}\}\cup
\{\overline{d_{k+1}},\dots,\overline{d_p}\}\bigr).\]

Since for $\ell=k+1,\ldots, p$ 
\[
\overline{e_{j_{\ell}}^{+}}+ \overline {d_{\ell}}=\overline{e_{t_{\ell}}^{+}}, \text{~in~} N_{k+1},
\] 
by lifting to $N$ under the canonical projection map, we obtain 
\[
	e_{j_{\ell}}^{+}+d_{\ell}=e_{t_{\ell}}^{+}+\sum_{j=1}^{j_{k+1}-1}a_j e_j^{+}, \qquad a_j\in\Z.
\] 
Since $j_{k+1}\le j_{\ell}<t_{\ell},$ comparing we get $a_j=0$ for all $1\le j \leq j_{k+1}-1$. Hence $e_{j_{\ell}}^{+}+d_{\ell}=e_{t_\ell}^{+}$.

This proves the lemma.
\end{proof}

We are now ready to prove the main result of the paper.

\begin{theorem}
\label{thm:direct-smoothness}
Let $w=s_{i_1}\cdots s_{i_r}\in W^P$ be a Coxeter-type element, and let
\[
J_P(\mathbf w):=\{\,j\in\{1,\dots,r\}\mid \alpha_{i_j}\in S_P\,\}
=\{j_1<\cdots<j_p\}.
\]
Then the toric Schubert variety $X_{wP}$ is smooth if and only if there exist pairwise distinct
integers
\[
t_1,\dots,t_p
\]
such that for every $k=1,\dots,p$,
\[
j_k<t_k\le r,
\qquad
\beta_{i_{j_k},i_{t_k}}=-1,
\qquad
\beta_{i_{j_k},i_m}=0\quad\text{for all } m>j_k,\ m\ne t_k.
\]
Equivalently, for every $k$, we have 
\[
e_{j_k}^+ + e_{j_k}^- = e_{t_k}^+,
\]
and the primitive generators of the cone $C_1^P$ are exactly
\[
\{e_i^+\mid i\notin\{t_1,\dots,t_p\}\}\cup\{e_{j_1}^-,\dots,e_{j_p}^-\}.
\]
In particular, $C_1^P$ is smooth.
\end{theorem}

\begin{proof}
For $1\le k\le p$, set
\[
d_k:=e_{j_k}^-.
\]
Since $w$ is of Coxeter type, Proposition~\ref{prop:toric} shows that $X_{wP}$ is toric. Hence
$X_{wP}$ is smooth if and only if the cone $C_1^P$ is smooth.

By Proposition~\ref{cone},
\[
C_1^P=
\mathrm{Cone}\bigl(\{e_1^+,\dots,e_r^+\}\cup\{d_1,\dots,d_p\}\bigr).
\]
For each $k=1,\dots,p$, Theorem~\ref{generatingmatrix} gives
\[
d_k=-e_{j_k}^+-\sum_{m>j_k}\beta_{i_{j_k},i_m}e_m^+.
\]
Since $w$ is of Coxeter type, the simple roots $\alpha_{i_1},\dots,\alpha_{i_r}$ are pairwise
distinct, so
\[
\beta_{i_{j_k},i_m}\le 0
\qquad\text{for all }m>j_k.
\]
Equivalently,
\[
d_k=-e_{j_k}^+ + \sum_{m>j_k} c_{k,m}e_m^+,
\qquad
c_{k,m}:=-\beta_{i_{j_k},i_m}\in\Z_{\ge 0}.
\]

\medskip

\noindent
$(\Rightarrow)$
Assume that $X_{wP}$ is smooth. We construct the integers $t_p,t_{p-1},\dots,t_1$ by descending
induction on $k$.

For $1\le k\le p$, let $\overline{C}^{(k)}_1$ be the smooth cone from
Lemma~\ref{lem:tail-cones}. We first treat the base case $k=p$.
The cone $\overline{C}^{(p)}_1$ is generated by
\[
\overline{e_{j_p}^+},\overline{e_{j_p+1}^+},\dots,\overline{e_r^+},\overline{d_p}.
\]
Since there is only one negative generator, Lemma~\ref{lemma:one-variable} applied to the smooth
toric Schubert variety $X_{w^{(p)}P}$ yields a unique index $t_p>j_p$ such that
\[
\overline{e_{j_p}^+}+\overline{d_p}=\overline{e_{t_p}^+}.
\]
This equality in $N_p$ lifts to
\[
e_{j_p}^+ + d_p = e_{t_p}^+ + \sum_{q<j_p} a_q e_q^+
\qquad\text{with } a_q\in \Z.
\]
Comparing this with
\[
e_{j_p}^+ + d_p = \sum_{m>j_p} c_{p,m}e_m^+,
\]
forces all $a_q=0$, and we obtain
\[
c_{p,t_p}=1,
\qquad
c_{p,m}=0\quad\text{for }m>j_p,\ m\ne t_p.
\]
Equivalently,
\[
\beta_{i_{j_p},i_{t_p}}=-1,
\qquad
\beta_{i_{j_p},i_m}=0\quad\text{for }m>j_p,\ m\ne t_p.
\]

Now fix $k<p$, and assume that $t_{k+1},\dots,t_p$ have already been constructed, are pairwise
distinct, and satisfy
\[
\overline{d_\ell}=\overline{e_{t_\ell}^+}-\overline{e_{j_\ell}^+}
\qquad\text{for }\ell=k+1,\dots,p, \text{~in~} N_{k+1}. 
\]
By Lemma~\ref{lem:successive-replacement}, the set
\[
B_{k+1}:=
\{\overline{e_m^+}\mid j_{k+1}\le m\le r,\ m\notin\{t_{k+1},\dots,t_p\}\}
\cup
\{\overline{d_{k+1}},\dots,\overline{d_p}\}
\]
is a $\Z$-basis of $N_{k+1}$, and
\[
\mathrm{Cone}(\overline{C}^{(k+1)}_1)
=
\mathrm{Cone}\bigl(\{\overline{e_{m}^+}: j_{k+1}\le m\le r, m\notin \{t_{k+1},\ldots, t_{p}\}\}\cup
\{\overline{d_{k+1}},\dots,\overline{d_p}\}\bigr)
\]
and the relations lift to
\[
d_\ell=e_{t_\ell}^+ -e_{j_\ell}^+
\qquad\text{for }\ell=k+1,\dots,p
\]
in $N$. 
Recall that
\[
\overline{C}^{(k)}_1=\mathrm{Cone}\bigl(\overline{e_{j_{k}}^+},\dots,\overline{e_r^+},
\overline{d_{k}},\dots,\overline{d_p}\bigr).
\]
Using the lifted relations, this smooth cone is generated by
\[
B_{k}=\{\overline{e_{j_{k}}^+},\dots,\overline{e_{j_{k+1}-1}^+}\} \cup \{\overline{e_m^+}\mid j_{k+1}\le m\le r,\ m\notin\{t_{k+1},\dots,t_p\}\}
\cup
\{\overline{d_{k}},\dots,\overline{d_p}\}.\]

We claim that $\overline{e_{j_k}^+}$ and each $\overline{d_\ell}$, $k\le \ell\le p$, span extremal rays of
$\overline{C}^{(k)}_1$.
The vector $\overline{d_k}$ is extremal because it is the only generator whose coefficient along
$\overline{e_{j_k}^+}$ is $-1$.
Also, $\overline{e_{j_k}^+}$ is extremal: indeed, every element of $\overline{C}^{(k)}_1$ other than itself and $\overline{d_k}$ has
$\overline{e_{j_k}^+}$-coefficient equal to $0$, whereas $\overline{d_k}$ has coefficient $-1$.
Therefore $\overline{e_{j_k}^+}$ cannot be written as a nonnegative linear combination of the
other generators. For $\ell=k+1,\dots,p$, the lifted relations
\[
d_\ell=e_{t_\ell}^+- e_{j_\ell}^+\qquad\text{in }N
\]
and the pairwise distinctness of $t_{k+1},\dots,t_p$ show that $\overline{d_\ell}$ spans an extremal ray of
$\overline{C}^{(k)}_1$, because its image in $\overline{C}_1^{(k+1)}$ does so.

Since $\overline{C}^{(k)}_1$ is smooth of dimension $\operatorname{rank}(N_k)$ and is generated by the
$\operatorname{rank}(N_k)+1$ vectors in  $B_k$,
its primitive ray generators form a
$\Z$-basis obtained by removing exactly one vector from $B_{k}$ and keeping
$\overline{d_\ell}$ for $\ell=k,\ldots, p$. Since $\overline{e_{j_k}^+}$, $\overline{d_k},\ldots, \overline{d_{p-1}}$ and $\overline{d_{p}}$ are extremal rays, the omitted vector must be some
\[
b_0\in B_{k}\setminus\{\overline{e_{j_k}^+}, \overline{d_k},\ldots, \overline{d_p}\},
\]
hence $b_0=\overline{e_{t_{k}}^{+}}$ for a unique index $t_{k}$ with $j_{k}<t_{k}$. Note that $t_{k}\notin \{t_{k+1},\ldots, t_{p}\}.$
Thus
\[
\mathcal B_k:= B_{k}\setminus\{b_0\}
\]
is a $\Z$-basis of $N_k$ consisting of the primitive ray generators of $\overline{C}^{(k)}_1$.

Since $b_0\in \overline{C}^{(k)}_1$, we can express it in the basis $\mathcal B_k$ as
\[
b_0
=
\sum_{\ell=k}^{p}a_{\ell}\,\overline{d_\ell}
+
c\,\overline{e_{j_{k}}^+}
+
\sum_{b\in B_{k}\setminus\{\overline{e_{j_k}^+},\,b_0,\, \overline{d_k}\,\ldots, \overline{d_p}\}}
a_b\, b,
\qquad
a_{\ell},c,a_b\in\Z_{\ge 0}.
\]
Since
\[
\overline{e_{j_k}^+}+\overline{d_k}
=
\sum_{m>j_k} c_{k,m}\overline{e_m^+},
\] and \[
\overline{e_{j_{\ell}}^+}+\overline{d_{\ell}}
=\overline{e_{t_{\ell}}^+},\, \text{~for all~} \ell=k+1,\ldots, p,
\] substituting the expression for $\overline{d_k}, \overline{d_{k+1}},\ldots, \overline{d_{p}}$, we obtain
\[b_0= a_{k}(-\overline{e_{j_{k}}^{+}}+\sum_{m>j_{k}}c_{k,m}\overline{e_{m}^{+}})
+\sum_{\ell=k+1}^{p}a_{\ell}\,(-\overline{e_{j_\ell}^{+}}+\overline{e_{t_{\ell}}^{+}})
+
c\,\overline{e_{j_{k}}^+}
+
\sum_{b\in B_{k}\setminus\{\overline{e_{j_k}^+},\,b_0,\, \overline{d_k},\,\ldots, \overline{d_p}\}}
a_b\, b,\]
where $a_{\ell},c,a_b\in\Z_{\ge 0}$.
Noting that $t_{k+1},\ldots , t_{p}$ are pairwise distinct and comparing coefficients, we obtain
\[
0=-a_{k}+c,\qquad
1=a_{k}c_{k,t_{k}},\qquad 
\qquad
0=a_kc_{k,t_{\ell}}+a_{\ell}\quad\text{for all}~ \ell=k+1,\ldots,p,\]
\[0=a_b+a_{k}c_{k,m}
\quad \text{for all} b=\overline{e_{m}^{+}}\in B_{k}\setminus\{\overline{e_{j_k}^+},\ldots, \overline{e_{j_p}^+}, b_0, \overline{d_k},\ldots, \overline{d_{p}}\}\]
\[0=a_kc_{k,j_\ell}-a_{\ell} +a_{\overline{e_{j_{\ell}}^{+}}}
\quad\text{for all } \ell=k+1,\ldots,p.\]

Hence $a_{k}=c=1$, $c_{k,t_{k}}=1$, $a_{b}=0$ for all
$b=\overline{e_{m}^+}\in B_{k}\setminus\{\overline{e_{j_k}^+},\ldots, \overline{e_{j_p}^+}, b_0, \overline{d_k},\ldots, \overline{d_{p}}\}$,
$a_{\ell}=0$ for all $\ell=k+1,\ldots,p$ and 
$c_{k,\ell}=0$ for all $\ell\neq t_{k}$. 

Therefore
\[
\overline{e_{j_k}^+}+\overline{d_k}=b_0=\overline{e_{t_k}^+}.
\]
Again this lifts to
\[
e_{j_k}^+ + d_k = e_{t_k}^+,
\]
Equivalently,
\[
\beta_{i_{j_k},i_{t_k}}=-1,
\qquad
\beta_{i_{j_k},i_m}=0\quad\text{for }m>j_k,\ m\ne t_k.
\]
Proceeding downward from $k=p$ to $k=1$, we obtain pairwise distinct integers
$t_1,\dots,t_p$ satisfying the stated conditions.

\medskip

\noindent
$(\Leftarrow)$
Conversely, assume that there exist pairwise distinct integers $t_1,\dots,t_p$ such that for every
$k$,
\[
e_{j_k}^+ + d_k = e_{t_k}^+,\quad \text{ or, equivalently,} ~ d_k=e_{t_k}^+-e_{j_k}^+.
\]

We claim that
\[
\mathcal B:= \{e_i^+\mid i\notin\{t_1,\dots,t_p\}\}\cup\{d_1,\dots,d_p\}
\]
is a $\Z$-basis of the lattice.
Indeed, start from the standard basis
\[
\{e_1^+,e_2^+,\dots,e_r^+\}
\]
and perform the replacements
\[
e_{t_p}^+\rightsquigarrow d_p,\quad
e_{t_{p-1}}^+\rightsquigarrow d_{p-1},\quad
\dots,\quad
e_{t_1}^+\rightsquigarrow d_1.
\]

Hence each replacement step is an elementary unimodular operation, so $\mathcal B$ is a
$\Z$-basis of $N$.
Moreover, at each step the generated cone is unchanged, because
\[
e_{t_k}^+=d_k+e_{j_k}^+
\]
is a nonnegative linear combination of the new generators.
Therefore the primitive ray generators of $C_1^P$ are exactly
\[
\{e_i^+\mid i\notin\{t_1,\dots,t_p\}\}\cup\{d_1,\dots,d_p\}.
\]
Since these generators form a $\Z$-basis of $N$, the cone $C_1^P$ is smooth. Hence $X_{wP}$ is
smooth.
\end{proof}

\begin{example}
In type $B_2$, let $P=P_{\widehat{\alpha_1}}$ be the maximal parabolic subgroup corresponding to the first simple root $\alpha_1$ and let $w=s_2s_1\in W^P$.
Then $[e,w]_{W^P}$ is again a chain with three elements.
However, here $\beta_{i_1,i_2}=-2$, so the one-variable smoothness criterion (Lemma~\ref{lemma:one-variable}) fails, and $X_{wP}$ is singular.
\end{example}

\begin{example}
In type $B_2$, let $P=P_{\widehat{\alpha_2}}$ be the maximal parabolic subgroup corresponding to the second simple root $\alpha_2$ and let $w=s_1s_2\in W^P$.
Then $[e,w]_{W^P}$ is again a chain with three elements.
However, here $\beta_{i_1,i_2}=-1$, so the one-variable smoothness criterion (Lemma~\ref{lemma:one-variable}) holds, and $X_{wP}$ is smooth.
\end{example}

\begin{remark}
\label{rem:inversion-smoothness}
The numerical condition of Theorem~\ref{thm:direct-smoothness} is stable under reversing the
chosen reduced word. Indeed, if one replaces
\[
w=s_{i_1}\cdots s_{i_r}
\]
by the reversed reduced word
\[
w^{-1}=s_{i_r}\cdots s_{i_1},
\]
then the indices are reversed and the equalities
\[
e_{j_k}^+ + e_{j_k}^- = e_{t_k}^+
\]
transform into the corresponding reversed equalities. Thus the criterion itself is combinatorially
symmetric under reversal of the reduced expression.

We do not formulate this as a statement about $X_{w^{-1}P}$ for a fixed parabolic $P$, since in
general $w^{-1}$ need not belong to $W^P$.
\end{remark}

\begin{definition}
Let $\mathbf{w}=s_{i_1}\cdots s_{i_r}$ be a reduced expression of a Coxeter-type element.
Define a directed graph $\Gamma(\mathbf{w})$ on the vertex set $\{1,\dots,r\}$ by putting an arrow
\[
j\longrightarrow m
\]
whenever $j<m$ and $\beta_{i_j,i_m}=-1$.
Let
\[
J_P(\mathbf{w})=\{j\mid \alpha_{i_j}\in S_P\}.
\]
\end{definition}

\begin{corollary}[Simply-laced criterion]\label{cor:simply-laced}
Assume that $G$ is simply laced, and let
\[
w=s_{i_1}\cdots s_{i_r}\in W^P
\]
be a Coxeter-type element. Let $\Gamma(\mathbf{w})$ be the directed graph on
$\{1,\dots,r\}$ defined by
\[
j\to m
\quad\Longleftrightarrow\quad
j<m \ \text{ and }\ \beta_{i_j,i_m}=-1.
\]
Then the toric Schubert variety $X_{wP}$ is smooth if and only if every vertex
$j\in J_P(\mathbf{w})$ has exactly one outgoing edge in $\Gamma(\mathbf{w})$, and the targets of these outgoing edges are pairwise distinct.
\end{corollary}

\begin{proof}
Since $G$ is simply laced, all off-diagonal Cartan integers belong to the set $\{0,-1\}$.
In particular, for $j<m$ we have
\[
\beta_{i_j,i_m}=-1
\quad\Longleftrightarrow\quad
j\to m \text{ in } \Gamma(\mathbf{w}),
\]
and otherwise $\beta_{i_j,i_m}=0$.

Assume first that $X_{wP}$ is smooth. Since $w$ is Coxeter-type, Theorem~\ref{thm:direct-smoothness} applies. 
Hence, for every $j\in J_P(\mathbf{w})$ there exists a unique index $\tau(j)>j$ such that
\[
\beta_{i_j,i_{\tau(j)}}=-1
\quad\text{and}\quad
\beta_{i_j,i_m}=0
\ \text{for all } m>j,\ m\neq \tau(j).
\]
Equivalently, the vertex $j$ has exactly one outgoing edge in $\Gamma(\mathbf{w})$, that is, 
\[
j\to \tau(j).
\]
Moreover, Theorem~\ref{thm:direct-smoothness} requires the indices $\tau(j)$ to be pairwise distinct. Thus the outgoing edges from vertices in $J_P(\mathbf{w})$ have pairwise distinct targets.

Conversely, assume that every vertex $j\in J_P(\mathbf{w})$ has exactly one outgoing edge in $\Gamma(\mathbf{w})$, and that the corresponding targets are pairwise distinct. For each such $j$, let $\tau(j)$ denote the unique target of the outgoing edge from $j$. Then
\[
j<\tau(j),
\qquad
\beta_{i_j,i_{\tau(j)}}=-1,
\]
and for every $m>j$ with $m\neq \tau(j)$, the absence of any other outgoing edge from $j$ means that
\[
\beta_{i_j,i_m}\neq -1.
\]
Since $G$ is simply laced, this forces
\[
\beta_{i_j,i_m}=0.
\]
Therefore the hypotheses of Theorem~\ref{thm:direct-smoothness} are satisfied, and we conclude that $X_{wP}$ is smooth.
\end{proof}

\begin{example}\label{ex:two-outgoing-singular}
Let $G=SL_4$, so the Dynkin diagram is of type $A_3$:
\[
\begin{tikzpicture}[
    scale=1,
    every node/.style={circle, draw, minimum size=.5mm},
    label distance=.5mm
]
    \node (A1) at (0,0) [label=below:$\alpha_1$] {};
    \node (A2) at (2,0) [label=below:$\alpha_2$] {};
    \node (A3) at (4,0) [label=below:$\alpha_3$] {};

    \draw (A1) -- (A2);
    \draw (A2) -- (A3);

\end{tikzpicture}
\]
Let $P=P_{\alpha_2}$ be the minimal parabolic subgroup corresponding to the simple root $\alpha_2$. Consider the Coxeter-type element  $w=s_2s_1s_3\in W^P$. We choose the reduced expression $\mathbf{w}=s_{i_1}s_{i_2}s_{i_3}=s_2s_1s_3.$
Then
\[
S_P=\{\alpha_2\},
\qquad\text{hence}\qquad
J_P(\mathbf{w})=\{1\}.
\]
Since the only position in the reduced word $\mathbf{w}$ where a simple root from $S_P$ occurs is the first place.
Next, we compute the relevant Cartan integers:
\[
\beta_{i_1,i_2}
=
\langle \alpha_2^\vee,\alpha_1\rangle
=
-1,
\qquad
\beta_{i_1,i_3}
=
\langle \alpha_2^\vee,\alpha_3\rangle
=
-1.
\]
Therefore, in the directed graph $\Gamma(\mathbf{w})$, the vertex $1$ has two outgoing edges:
\[
1\to 2,
\qquad
1\to 3.
\]
Since $1\in J_P(\mathbf{w})$, the condition in Theorem~\ref{thm:direct-smoothness} fails. 
In other words, there is no unique index
$t>1$ such that
\[
\beta_{i_1,i_t}=-1
\quad\text{and}\quad
\beta_{i_1,i_m}=0 \ \text{for all } m>1,\ m\neq t.
\]
Hence the toric Schubert variety $X_{wP}$ is singular.

This example shows that, even in simply-laced type, a toric Schubert variety in $G/P$ can be singular as soon as a vertex belonging to $J_P(\mathbf{w})$ has more than one outgoing edge in $\Gamma(\mathbf{w})$.
\end{example}

\begin{example}\label{ex:two-incoming-singular}
Let $G=SL_4$, so the Dynkin diagram is of type $A_3$:
\[
\begin{tikzpicture}[
    scale=1,
    every node/.style={circle, draw, minimum size=.5mm},
    label distance=.5mm
]
    \node (A1) at (0,0) [label=below:$\alpha_1$] {};
    \node (A2) at (2,0) [label=below:$\alpha_2$] {};
    \node (A3) at (4,0) [label=below:$\alpha_3$] {};

    \draw (A1) -- (A2);
    \draw (A2) -- (A3);

\end{tikzpicture}
\]
Let $P=P_{\widehat{\alpha_2}}$ be the maximal parabolic subgroup corresponding to the simple root $\alpha_2$. Consider the Coxeter-type element  $w=s_1s_3s_2\in W^P$. We choose the reduced expression $\mathbf{w}=s_{i_1}s_{i_2}s_{i_3}=s_1s_3s_2.$
Then
\[
S_P=\{\alpha_1, \alpha_3\},
\qquad\text{hence}\qquad
J_P(\mathbf{w})=\{1,2 \}.
\]
Since the positions in the reduced word $\mathbf{w}$ where simple roots from $S_P$ occur are the first and second positions.
Next, we compute the relevant Cartan integers:
\[
\beta_{i_1,i_3}
=
\langle \alpha_1^\vee,\alpha_2\rangle
=
-1,
\qquad
\beta_{i_2,i_3}
=
\langle \alpha_3^\vee,\alpha_2\rangle
=
-1.
\]
Therefore, in the directed graph $\Gamma(\mathbf{w})$, the vertex $3$ has two incoming edges:
\[
1\to 3,
\qquad
2\to 3.
\]
Since $1, 2\in J_P(\mathbf{w})$, the condition in Theorem~\ref{thm:direct-smoothness} fails. 
In other words, there do not exist distinct indices
$t_1>1$ and $t_2>2$ such that
\[
\beta_{i_1,i_{t_1}}=-1
\quad\text{and}\quad
\beta_{i_1,i_m}=0 \ \text{for all } m>1,\ m\neq t_1.
\]
and
\[
\beta_{i_2,i_{t_2}}=-1
\quad\text{and}\quad
\beta_{i_2,i_m}=0 \ \text{for all } m>2,\ m\neq t_2.
\]
Hence the toric Schubert variety $X_{wP}$ is singular.

This example shows that, even in simply-laced type, a toric Schubert variety in $G/P$ can be singular as soon as distinct vertices belonging to $J_P(\mathbf{w})$ have outgoing edges with the same target vertex in $\Gamma(\mathbf{w})$.
\end{example}

\section{Applications to Spherical Schubert Varieties in \texorpdfstring{$G/B$}{G/B}}
\label{sec:applications}

We now relate our smoothness criteria for toric Schubert varieties to the geometry of spherical Schubert varieties.

We begin with the horospherical case, which provides a simple application of our results.

\begin{corollary}
Every horospherical Schubert variety $X_{wB} \subset G/B$ is smooth.
\end{corollary}

\begin{proof}
Let $X_{wB}$ be horospherical with respect to a Levi subgroup $L_J$. Then by \cite{CS2023}, we have a decomposition
\[
w = w_{0,J} c, \qquad \ell(w) = \ell(w_{0,J}) + \ell(c),
\]
where $c$ is a Coxeter-type element with $\operatorname{supp}(c) \cap J = \emptyset$.
By Carrell's result~\cite[Corollary 4]{Carrell2011}, $X_{wB}$ is smooth if and only if $X_{w^{-1}B}$ is smooth. Since
$w^{-1} = c^{-1} w_{0,J}$, 
the projection $\pi : G/B \to G/P_J$ satisfies
\[
\pi(X_{w^{-1}B}) = X_{c^{-1}P_J}.
\]
Moreover, $\pi$ is smooth with fiber $X_{w_{0,J}}$, which is smooth. Hence $X_{w^{-1}B}$ is smooth if and only if $X_{c^{-1}P_J}$ is smooth.
Since $\operatorname{supp}(c) \cap J = \emptyset$, Proposition~\ref{prop:smoothat1P} implies that $X_{c^{-1}P_J}$ is smooth. Therefore $X_{wB}$ is smooth.
\end{proof}

We now consider the general spherical case, especially Schubert varieties that are spherical with respect to a Levi subgroup. 
For background and related results on Levi-spherical Schubert varieties, see \cite{ CS2023, GaoHodgesYong}.

\begin{theorem}\label{thm:spherical-smoothness}
Let $w\in W$ and suppose that $X_{wB}$ is spherical with respect to a Levi subgroup $L_J$. Write
$$w = w_{0,J} c, \qquad \ell(w) = \ell(w_{0,J}) + \ell(c),$$
where $c$ is a Coxeter-type element.

Then the following are equivalent:
\begin{enumerate}
\item $X_{wB}$ is smooth;
\item $X_{w^{-1}B}$ is smooth;
\item the toric Schubert variety $X_{c^{-1}P_J}$ is smooth (cf. Theorem~\ref{thm:direct-smoothness}).
\end{enumerate}
\end{theorem}

\begin{proof}
The equivalence $(1) \iff (2)$ follows from~\cite[Corollary 4]{Carrell2011}.

For $(2) \iff (3)$, we consider the inverse element $w^{-1} = c^{-1} w_{0,J}$. Since lengths are additive, $\ell(w^{-1}) = \ell(c^{-1}) + \ell(w_{0,J})$. Because $c^{-1} \in W^J$ and $w_{0,J} \in W_J$, this gives the standard parabolic decomposition of $w^{-1}$ with respect to $J$.

We now apply the work of Richmond and Slofstra~\cite{RichmondSlofstra} on the Billey-Postnikov (BP) decompositions. 
In their notation, we set their $J = \emptyset$ and their $K = J$. By~\cite[Proposition 4.2(c)]{RichmondSlofstra}, a parabolic decomposition $x = vu$ with respect to $K$ is a BP decomposition if $u$ is the maximal element of $[e, x] \cap W_K^\emptyset$. Substituting our elements ($x = w^{-1}$, $v = c^{-1}$, and $u = w_{0,J}$), we see this condition is trivially satisfied because $w_{0,J}$ is the maximal element of the entire parabolic subgroup $W_J$. Thus, $w^{-1} = c^{-1} w_{0,J}$ is a valid BP decomposition.

Consequently, by~\cite[Theorem 3.3]{RichmondSlofstra}, the projection $\pi: X_{w^{-1}B} \to X_{c^{-1}P_J}$ is a Zariski-locally trivial fiber bundle with fiber $X_{w_{0,J}B}$. Furthermore, part (1) of~\cite[Theorem 3.3]{RichmondSlofstra} dictates that the total space $X_{w^{-1}B}$ is smooth if and only if both the base $X_{c^{-1}P_J}$ and the fiber $X_{w_{0,J}B}$ are smooth. Since the fiber $X_{w_{0,J}B}$ is isomorphic to the full flag variety of the Levi subgroup, it is intrinsically smooth. Therefore, $X_{w^{-1}B}$ is smooth if and only if $X_{c^{-1}P_J}$ is smooth.
This finishes the proof of our theorem. 
\end{proof}

In light of Theorem~\ref{thm:direct-smoothness}, our previous result provides an effective combinatorial way of determining which spherical Schubert varieties are smooth in all types.

\bibliographystyle{plain}
\bibliography{references_final}
\end{document}